\documentclass[UTF-8,reqno]{amsart}
\usepackage{enumerate}
\usepackage{mhequ}
\usepackage{amssymb,url,color, booktabs,nccmath}
\usepackage[left=3.2cm,right=3.2cm,top=4cm,bottom=4cm]{geometry}
\usepackage{mathrsfs}
\usepackage{enumitem,dsfont}

\usepackage{color}
\usepackage[colorlinks=true]{hyperref}
\hypersetup{
    linkcolor=blue,          
    citecolor=red,        
    filecolor=blue,      
    urlcolor=cyan
}

\definecolor{darkergreen}{rgb}{0.0, 0.5, 0.0}

\numberwithin{equation}{section}
\def\theequation{\arabic{section}.\arabic{equation}}
\newcommand{\be}{\begin{eqnarray}}
\newcommand{\ee}{\end{eqnarray}}
\newcommand{\ce}{\begin{eqnarray*}}
\newcommand{\de}{\end{eqnarray*}}
\newtheorem{theorem}{Theorem}[section]
\newtheorem*{theorem*}{Theorem}
\newtheorem{lemma}[theorem]{Lemma}
\newtheorem{proposition}[theorem]{Proposition}
\newtheorem{Examples}[theorem]{Example}
\newtheorem{corollary}[theorem]{Corollary}

\newtheorem{definition}[theorem]{Definition}
\theoremstyle{definition}
\newtheorem{remark}[theorem]{Remark}

\newcommand{\cdummy}{\cdot}

\newcommand{\tmop}[1]{\ensuremath{\operatorname{#1}}}

\def\${|\!|\!|}

\DeclareMathOperator{\supp}{supp}

\def\eps{\varepsilon}

\def\p{\partial}

\def\<{{\langle}}
\def\>{{\rangle}}
\def\({{\Big(}}
\def\){{\Big)}}

\def\bx{{\mathbf{x}}}

\def\dif{{\mathord{{\rm d}}}}

\def\no{\nonumber}
\def\={&\!\!=\!\!&}

\def\cF{{\mathcal F}}

\def\cL{{\mathcal L}}

\def\cT{{\mathcal T}}

\def\mN{{\mathbb N}}

\def\mR{{\mathbb R}}

\def\mT{{\mathbb T}}

\def\mZ{{\mathbb Z}}

\def\bP{{\mathbf P}}

\def\1{{\mathbf{1}}}

\def\E{\mathbf E}

\def\geq{\geqslant}
\def\leq{\leqslant}

\def\div{\mathord{{\rm div}}}

\def\eps{\varepsilon}

\def\p{\partial}

\def\<{{\langle}}
\def\>{{\rangle}}
\def\({{\Big(}}
\def\){{\Big)}}

\def\bx{{\mathbf{x}}}

\def\dif{{\mathord{{\rm d}}}}

\def\no{\nonumber}
\def\={&\!\!=\!\!&}
\def\bt{\begin{theorem}}
\def\et{\end{theorem}}
\def\bl{\begin{lemma}}
\def\el{\end{lemma}}
\def\br{\begin{remark}}
\def\er{\end{remark}}
\def\bx{\begin{Examples}}
\def\ex{\end{Examples}}
\def\bd{\begin{definition}}
\def\ed{\end{definition}}
\def\bp{\begin{proposition}}
\def\ep{\end{proposition}}
\def\bc{\begin{corollary}}
\def\ec{\end{corollary}}

\def\geq{\geqslant}
\def\leq{\leqslant}

\def\div{\mathord{{\rm div}}}

\def\bP{{\mathbf P}}

 \def\R{\mathbb R}
 \def\R{\mathbb R}    
\def\N{\mathbb N}  
   
\def\<{\langle} \def\>{\rangle}

\allowdisplaybreaks

\begin{document}

\title[SQG equation perturbed by derivatives of space-time white noise]{Surface quasi-geostrophic equation perturbed by  derivatives of space-time white noise}

\author{Martina Hofmanov\'a}
\address[M. Hofmanov\'a]{Fakult\"at f\"ur Mathematik, Universit\"at Bielefeld, D-33501 Bielefeld, Germany}
\email{hofmanova@math.uni-bielefeld.de}

\author{Xiaoyutao Luo}
\address[X. Luo]{Academy of Mathematics and Systems Science,
	Chinese Academy of Sciences, Beijing 100190, China}
\email{xiaoyutao.luo@amss.ac.cn}

\author{Rongchan Zhu}
\address[R. Zhu]{Department of Mathematics, Beijing Institute of Technology, Beijing 100081, China}
\email{zhurongchan@126.com}

\author{Xiangchan Zhu}
\address[X. Zhu]{ Academy of Mathematics and Systems Science,
Chinese Academy of Sciences, Beijing 100190, China}
\email{zhuxiangchan@126.com}
\thanks{
R.Z. and X.Z. are grateful for
the financial supports   by National Key R\&D Program of China (No. 2022YFA1006300).
R.Z. gratefully acknowledges financial support from the NSFC (No.
11922103, 12271030).
  X.Z. is grateful for the financial supports  in part by National Key R\&D Program of China (No. 2020YFA0712700) and the NSFC (No.   12090014, 12288201) and
  the support by key Lab of Random Complex Structures and Data Science,
 Youth Innovation Promotion Association (2020003), Chinese Academy of Science.
The research of  M.H.  was funded by the European Research Council (ERC) under the European Union's Horizon 2020 research and innovation programme (grant agreement No. 949981). The research of M.H., R.Z. and X.Z.  was funded by the Deutsche Forschungsgemeinschaft (DFG, German Research Foundation) -- SFB 1283/2 2021 -- 317210226.
}

\begin{abstract}
We consider a family of singular surface quasi-geostrophic equations
$$
\partial_{t}\theta+u\cdot\nabla\theta=-\nu (-\Delta)^{\gamma/2}\theta+(-\Delta)^{\alpha/2}\xi,\qquad u=\nabla^{\perp}(-\Delta)^{-1/2}\theta,
$$
on $[0,\infty)\times\mathbb{T}^{2}$, where $\nu\geq 0$, $\gamma\in [0,3/2)$, $\alpha\in [0,1/4)$ and $\xi$ is a space-time white noise.
For the first time, we establish the \emph{existence of infinitely many non-Gaussian}
\begin{itemize}
\item probabilistically strong solutions for every initial condition in $C^{\eta}$, $\eta>1/2$;
\item ergodic stationary solutions.
\end{itemize}
The result presents a single approach applicable in the subcritical, critical as well as supercritical regime in the sense of Hairer~\cite{Hai14}. It also applies  in the particular setting $\alpha=\gamma/2$ which formally possesses a Gaussian invariant measure.
In our proof, we first introduce a modified Da Prato--Debussche trick which, on the one hand, permits to convert irregularity in time into irregularity in space  and, on the other hand, increases the regularity of the linear solution. Second, we develop a convex integration iteration for the corresponding nonlinear equation which yields non-unique non-Gaussian solutions satisfying powerful global-in-time estimates and generating stationary as well as ergodic stationary solutions.
 \end{abstract}

\date{\today}

\maketitle

\tableofcontents

\section{Introduction}

We are interested in construction of solutions to the  surface quasi-geostrophic (SQG) equation  on $\mathbb{T}^2$ driven by a fractional derivative of a space-time white noise. More precisely,
\begin{equation}\label{eq2}
\begin{aligned}
\partial_{t }\theta+ u\cdot\nabla\theta &= -\nu\Lambda^{\gamma}\theta+\Lambda^\alpha \xi,\\
u=\nabla^\perp\Lambda^{-1}\theta&=(-\partial_2\Lambda^{-1}\theta,\partial_1\Lambda^{-1}\theta)=(-\mathcal{R}_2\theta, \mathcal{R}_1\theta),\\
\theta(0)&=\theta_{0},
\end{aligned}
\end{equation}
where $\Lambda=(-\Delta)^{1/2}$, $\nu\geq0$, $\gamma\in [0,3/2)$, $\alpha\in[0,1/4)$,  $\mathcal{R}=(\mathcal{R}_1,\mathcal{R}_2)$ is the pair of Riesz transforms and $\xi$ is a space-time white noise and can be written as time derivative of an $L^2$-cylindrical Wiener process $B$ on some stochastic basis $(\Omega,\mathcal{F},(\mathcal{F}_t)_{t\geq0},\mathbf{P})$. The unknown scalar field $\theta$ denotes for instance   temperature in the geophysical fluid dynamics, see e.g.~\cite{pedlosky2013geophysical}.

As in the deterministic literature~\cite{marchand2008existence} on the SQG equation, we rely on a weak formulation of the equation in $\dot{H}^{-1/2}$. In particular, we work with analytically weak solutions with zero spatial mean and we profit from the fact that the nonlinearity in \eqref{eq2} is well-defined for $\theta\in L^{2}_{\rm{loc}}([0,\infty);\dot{H}^{-1/2})$. Indeed, for any $\psi\in C^{\infty}$  it holds
\begin{equation}\label{eq:com}
\langle u\cdot\nabla\theta,\psi\rangle =\frac12 \langle \theta,[\mathcal{R}^{\perp}\cdot,\nabla\psi]\theta\rangle,
\end{equation}
where the commutator $[\mathcal{R}^{\perp}\cdot,\nabla\psi]=-[\mathcal{R}_2\cdot,\partial_1\psi]+[\mathcal{R}_1\cdot,\partial_2\psi]$ maps $\dot{H}^{-1/2}$ to $\dot{H}^{1/2}$ (c.f. Proposition~\ref{p:51}). Hence,
we aim at solving \eqref{eq2} in the following sense.

\begin{definition}\label{d:1}
  We say that $\theta \in L^{2}_{\rm{loc}}([0,\infty);\dot{H}^{-1/2})\cap C([0,\infty);B^{-1}_{\infty,1})$ is an analytically weak solution to \eqref{eq2}
  provided $\theta$ is $(\mathcal{F}_t)$-adapted and
for any $t\geq0$
\begin{align*}
\langle \theta(t),\psi\rangle+\int_0^t\frac12\langle \Lambda^{-1/2}\theta,\Lambda^{1/2}[\mathcal{R}^\perp\cdot,\nabla \psi] \theta\rangle\dif t
 = \langle \theta_0,\psi\rangle - \int_{0}^{t}\langle\nu\Lambda^{\gamma } \theta
,\psi\rangle\dif t+\int_0^t\langle\Lambda^\alpha\psi,\dif B_t\rangle
\end{align*}
holds  for every $\psi \in C^{\infty}$.
\end{definition}

Our main result reads as follows. The reader is referred to Theorem~\ref{thm:6.1} and Theorem~\ref{thm:4.5} for the precise formulation and further details.

\bt\label{thm:1.2}
Let $\nu\geq 0$, $\gamma\in[0,3/2)$, $\alpha\in[0,1/4)$ and let $\xi$ be a space-time white noise. Then the SQG equation \eqref{eq2} possesses
\begin{enumerate}
\item infinitely many non-Gaussian analytically weak  probabilistically strong solutions for every  $\theta_{0}\in C^{\eta}$ with $\eta>1/2$;
\item infinitely many non-Gaussian ergodic stationary solutions.
\end{enumerate}
\et

\begin{remark}
The main contribution of Theorem \ref{thm:1.2} is as follows. We discuss further details in Section~\ref{s:1.2}.
	\begin{itemize}
		\item This is the first result of a non-unique non-Gaussian construction of global solutions to the SQG equation (both inviscid and viscous) in the presence of a \emph{space-time} white noise, cf.~\cite{HZZ22} where the case of only \emph{spatial} white noise was covered.
		
		\item To handle the irregularity in time of the space-time white noise,  we develop a modified Da~Prato--Debussche trick and a novel convex integration iteration that requires the use of \emph{all} previous induction estimates, see~Proposition \ref{p:iteration} below.

	\end{itemize}
\end{remark}

\subsection{SPDEs and criticality}

The  problem \eqref{eq2} falls in the category of singular stochastic partial differential equations (SPDEs) where the nonlinearity is heuristically not well defined analytically  due to the irregularity of the noise $\xi$. This can be seen as follows. In the scaling of the equation, $\Lambda^{\alpha}\xi$ is a random distribution of space-time regularity $-1-\gamma/2-\alpha-\kappa$ for arbitrary $\kappa>0$. By Schauder's estimates in  the viscous setting    $\nu>0$ we expect a solution $\theta$ to be $\gamma$-degrees of regularity better, i.e. $-1+\gamma/2-\alpha-\kappa$. This regularity is always negative in our range of parameters and hence $\theta$ is only a distribution and the nonlinearity in the mild formulation does not make sense analytically. A further scaling argument then determines the so-called subcritical, critical and supercritical regimes as formulated by Hairer's  theory of regularity structures \cite{Hai14}. Depending on the particular value of the parameters $\nu$, $\gamma$, $\alpha$, the SQG equation \eqref{eq2} may belong to either one of these regimes. We also observe that the solutions constructed in Theorem \ref{thm:1.2} stay in $\dot{H}^{-1/2}$. This can potentially offer improved regularity compared to utilizing only Schauder's estimates when $\gamma$ is small.

The problem of local well-posedness of solutions in the subcritical regime is well-understood by Hairer's regularity structures theory \cite{Hai14} and the paracontrolled distributions method by Gubinelli, Imkeller and Perkowski  \cite{GIP15}.
These theories permit to treat a large number of singular subcritical SPDEs
(cf. \cite{BHZ19, CH16, BCCH21, OSSW21, HS23})
including the Kardar--Parisi--Zhang (KPZ) equation, stochastic quantization equations for quantum fields as well as the SQG equation driven by space-time white
	noise, i.e.  $\gamma>4/3$, $\alpha=0$, in \cite{FS21}.
		 Existence of global solutions typically relies on some special feature of the equation at hand, {like a strong drift, a particular transform or properties of an invariant measure. We refer to \cite{HZZ21,HZZ22} for a detailed discussion and further references.}

Additionally, in \cite{HZZ21},  non-unique global solutions to 3D Navier--Stokes equations driven by space--time white noise were constructed  via combination of paracontrolled calculus and convex integration. In the 2D case, global unique solutions to  Navier--Stokes equations driven by space--time white noise, with an additional perturbation in $C^{-2+\kappa}$  were obtained in \cite{HR23} and sharp non-uniqueness via convex integration method was established in \cite{LZ23}.

As previously discussed, the nonlinear term in the SQG equation \eqref{eq2} remains well-defined even when  $\theta\in \dot{H}^{-1/2}$, thereby giving rise to a concept of weak solution. In the deterministic scenario, the associated estimate for the $H^{-1/2}$-norm of the solution $\theta$ can be derived, resulting in the establishment of global-in-time weak solutions.
	However, in our current context, this approach is not applicable to infer global solutions due to the singular nature of the noise.

There are only a few results in the critical and supercritical setting. In the critical regime, extensive exploration of the 2D isotropic Kardar--Parisi--Zhang (KPZ) equation was done in  \cite{CD20, CSZ20, Gu20} by employing Cole--Hopf's transform. Meanwhile, the study of the 2D anisotropic KPZ equation was addressed in \cite{CES21, CET21}, building upon the underlying invariant measure given by a Gaussian distribution. Notable advancements also emerged in the domain of the critical 2D stochastic heat flow, as demonstrated by recent developments in \cite{CSZ21}.

	Transitioning to the supercritical regime, a recent series of works encompassing \cite{MU18, CCM20, DGRZ20, CNN20, CGT23}  delved into the KPZ equation or its derivative Burgers equation for dimensions $d\geq 3$. More specifically, these studies  comprehended the $d$-dimensional KPZ equation or Burgers equations as limits of suitably regularized equations. Notably, the regularization incorporated a vanishing scaling constant preceding the nonlinearity. This process culminated in a stochastic heat equation as the limiting equation, characterized by a Gaussian law, although exceptions are noted in \cite{CSZ21}.

In a related context, the work of \cite{GJ13} has achieved the construction of global stationary probabilistically weak solutions for the stochastic fractional Burgers equation, both within the critical and supercritical regimes by developing the so-called It\^{o}'s trick. In \cite{GP23} the authors sketched the well-posedness for the fractional Burgers equation in the critical case and the SQG equation \eqref{eq2} when $\gamma\geq 2$, $\alpha=\gamma/2$ by  constructing the infinitesimal generator. Moreover, using It\^{o}'s trick they could obtain existence of global stationary probabilistically weak solutions when $\gamma>0$, $\alpha=\gamma/2$. This accomplishment has been realized through the energy solution methodology, reliant on the associated invariant measure originating from the law of space white noise. If $\alpha\neq \gamma/2$ then this method
may fail.

Unlike the above mentioned techniques, our construction in the present paper is independent of the particular regime at hand and applies to the subcritical, critical as well as supercritical case. In our previous work \cite{HZZ22}, we obtained a first result in this direction in the case of purely spatial perturbation: $\xi$ was a space white noise and $\alpha\in [0,1)$.
The approach   made use of the convex integration, an iterative technique of direct construction of non-unique solutions. Even though  certain time dependence could be included in the noise, it would have to be regular, leaving the space-time white noise out of reach. In the present manuscript, we overcome these difficulties by introducing several new ideas.

\subsection{Outline of ideas and difficulties}\label{s:1.2}

Already in the deterministic situation, the first problem of trying to deploy a convex integration construction to the SQG equation is the odd structure of the SQG nonlinearity, as it was observed in~\cite{shvConvInt,isettVicol}. This means that the multiplier relating $\theta$ to $u$ is an odd function of frequency in Fourier space. This oddness relates to the boundedness of the nonlinearity in $\dot{H}^{-1/2}$, a key to constructing global weak solutions~\cite{marchand2008existence}.  However, it  remained open whether they are uniquely determined by their initial data.

The  breakthrough result~\cite{BSV19}  successfully overcame the issue of odd multiplier and constructed nontrivial high-high to low frequency interactions that led to the first non-uniqueness result for the SQG equation.  The convex integration scheme was designed for the associated momentum equation satisfied by the potential velocity $\Lambda^{-1}u$. Subsequent developments of \cite{CKL21} and \cite{IM21} provided different convex integration schemes to  construct non-unique weak solutions.  The technique developed in  \cite{IM21} works directly on the level of the scalar field $\theta$ and applies also to general odd multipliers. The construction of \cite{CKL21} is particularly appealing due to its simplicity: it applies to the scalar function $\Lambda^{-1}\theta$ and makes use of the special structure of the SQG nonlinearity, leading to a concise proof. We also mention that the method in \cite{BSV19} has been extended to the SQG equation driven by trace-class type additive noise in \cite{Y22}.

 {The convex integration scheme in the present paper is inspired by \cite{HZZ22} which in turn relied on the method of \cite{CKL21}.}  We work on the level of $v:=\Lambda^{-1}\theta$ and in the first step we decompose the equation with the help of a modified Da Prato--Debussche trick  with multiple benefits. Roughly speaking, we decompose $v = z+f$ where $z$ solves the linear SPDE
\begin{equation*}
\partial_{t}z = - \nu_{1}\Lambda^{\gamma_{1}} z +\Lambda^{-1+\alpha}\xi,
\end{equation*}
 {where $\nu_{1}>0$}, possibly  $\gamma_{1}\neq \gamma$ and $f$ solves the corresponding nonlinear equation (see Section~\ref{s:d} for more details).  {We set the  viscosity $\nu_{1}=1$ for simplicity, even in the inviscid setting  $\nu=0$.} The additional parameter $\gamma_{1}$ is particularly helpful if $\gamma$ and/or $\eta$ are small within their respective ranges. The conventional Da Prato--Debussche trick  refers to the case of  {$\nu_{1}=\nu>0$ and} $\gamma_{1}=\gamma$ and  it has been extensively applied in the literature in various subcritical settings (c.f. \cite{DPD03}). It is however not suitable for our purposes for all values of the parameters  {$\nu$,} $\gamma$, $\alpha$, $\eta$ due to the possibly critical/supercritical nature of the SQG equation \eqref{eq2}. For instance, if $\gamma$ is small then even by  convex integration we are not able to treat the terms containing $z$ in the equation for $f$.
By choosing $\gamma_{1}>\gamma$ the regularity of $z$ increases and  the control of the mentioned terms, even though still very delicate, becomes  feasible. We emphasize that as we are in the critical or supercritical regime, further decomposition like the Picard iteration does not provide any gain in regularity.

The price to pay for this modification of the Da Prato--Debussche trick is the additional term $\Lambda^{\gamma_{1}}z$ in the equation for $f$. Seemingly, there is no improvement as this term  formally has  the same spatial regularity as the original  noise term in the equation for $v$. However, the point is that the time irregularity is eliminated, so effectively the space-time irregularity of the noise was transformed into a spatial irregularity of $\Lambda^{\gamma_{1}}z$, bringing us closer to the convex integration technique developed in \cite{HZZ22}.
In the latter work, however, no probability theory was involved and the proof worked for arbitrary (deterministic) spatial distribution $\xi $ within the appropriate regularity class.

In the present paper, we have an additional difficult term, namely, the quadratic term $\Lambda z\nabla^{\perp}z$ in the equation for $f$. As $z$ only has the regularity $\gamma_{1}/2-\alpha-\kappa$ for arbitrary $\kappa>0$, the said product is not well defined analytically. Nevertheless, using Gaussianity of $z$ we may apply renormalization to define it as a Wick product $:\Lambda z\nabla^{\perp}z:$ (see Proposition~\ref{lem:re}). This is then employed in the equation for $f$ as a replacement for the ill-defined product. The next step is purely deterministic: via a convex integration iterative procedure we construct approximations $f_{n}$, $n\in\mathbb{N}_{0}$, and show that they converge to a limit $f$.  Finally, letting $\theta=\Lambda z + \Lambda f$ we obtain a solution to the SQG equation \eqref{eq2} in the sense of Definition~\ref{d:1}.

 In contrast to the deterministic framework of the convex integration method applied to the SQG equation, as outlined in  \cite{BSV19, CKL21, IM21}, it is also necessary to handle the additional nonlinear term that involves the process $z$. Unlike the lacunary Fourier modes of the building blocks defined in \cite{CKL21}, the variable $z$ lacks this property. Consequently, it becomes imperative to undertake a more intricate decomposition, akin to the dyadic partition technique employed in Fourier space, for these specific terms. This decomposition aims to ensure that the resulting components are either confined within a small annular region or localized within a  small ball.

We remark that even though the renormalized product $:\Lambda z\nabla^{\perp}z:$ is indispensable for the construction, the renormalization constant does not explicitly appear in the definition of  solution to \eqref{eq2}. This goes back to the form of the nonlinearity, in particular the antisymmetry of the Riesz transform and the formula \eqref{eq:com}.

Finally, we mention that this modified Da Prato--Debussche trick combined with convex integration also applies to  fractional Navier--Stokes equations in the supercritical and critical regime, see Section~\ref{other}.

\subsection{Organization of the paper}

In Section~\ref{s:not}, we summarize frequently used notations. Section~\ref{s:d} represents the first step in the construction: we introduce the modified Da Prato--Debussche trick and decompose the problem into a linear SPDE with irregular space-time noise and a nonlinear PDE with spatial irregularity. In Section~\ref{s:proba}, we present the probabilistic construction, namely, we employ renormalization to make sense of the analytically ill-defined product. Section~\ref{s:a} sets up the analytic construction via the convex integration iteration. The detailed construction is presented in Section~\ref{s:it}. Our main result on the initial value problem is established in Section~\ref{s:1.1}, whereas the result on non-unique ergodicity is shown in  Section~\ref{s:in}. Further extension to fractional Navier--Stokes system is presented  in Section \ref{other}. In Appendix~\ref{s:B}, we collect several auxiliary results.

\section{Notations}
\label{s:not}

\subsection{Function spaces and operators}

  Throughout the paper, we employ the notation $a\lesssim b$ if there exists a constant $c>0$ such that $a\leq cb$, and we write $a\simeq b$ if $a\lesssim b$ and $b\lesssim a$. $\mN_{0}:=\mN\cup \{0\}$. Given a Banach space $E$ with a norm  $\|\cdot\|_E$, we write  $CE$ or $C([0,\infty);E)$ to denote the space of continuous functions from $[0,\infty)$ to $E$.  We  define  $C^1E$ as the space of  continuous functions with locally bounded first order derivative from $[0,\infty)$ to $E$.  Similarly,
  for $\alpha\in(0,1)$ we  define $C^\alpha_TE$ as the space of $\alpha$-H\"{o}lder continuous functions from $[T,T+1]$  to $E$, endowed with the norm $\|f\|_{C^\alpha_TE}=\sup_{s,t\in[T,T+1],s\neq t}\frac{\|f(s)-f(t)\|_E}{|t-s|^\alpha}+\sup_{t\in[T,T+1]}\|f(t)\|_{E}$.
  For $p\in [1,\infty]$ we write $L^p_TE=L^p(0,T;E)$ for the space of $L^p$-integrable functions from $[0,T]$ to $E$, equipped with the usual $L^p$-norm. We also use $L^p_{\mathrm{loc}}([0,\infty);E)$ to denote the space of functions $f$ from $[0,\infty)$ to $E$ satisfying $f|_{[0,T]}\in L^p_T E$ for all $ T>0$.

We also introduce the following notations: the  Fourier and the inverse Fourier transform on $\mR^2$ are defined as
$$\mathcal{F}_{\mR^2}f(\zeta)=\int_{\mR^2}e^{-ix\cdot\zeta}f(x)\dif x,\quad \mathcal{F}_{\mR^2}^{-1}f(x)=\frac1{(2\pi)^2}\int_{\mR^2}e^{ix\cdot\zeta}f(\zeta)\dif \zeta.$$
The  Fourier and the inverse Fourier transform on $\mT^2$ are defined as
$$\mathcal{F}_{\mT^2}f(k)=\hat{f}(k)=\int_{\mT^2}e^{ix\cdot k}f(x)\dif x,\quad k\in\mathbb{Z}^2,\quad \mathcal{F}_{\mT^2}^{-1}f(x)=\frac1{(2\pi)^2}\sum_{k\in\mZ^2}e^{ix\cdot k}f(k).$$
A function on the torus can be viewed as a function on $\mR^2$ by periodic extension. Assume that $f\in C^\infty(\mT^2)$ and $m\in L^1(\mathbb{R}^2)$ and define
$$T_mf(x):=\mathcal{F}_{\mT^2}^{-1}m\mathcal{F}_{\mT^2}f(x)=\mathcal{F}_{\mR^2}^{-1}m\mathcal{F}_{\mR^2}f(x),$$
where the second equality follows from the fact that we view the function on $\mT^2$ as periodic function on $\mR^2$.
 This implies that
\begin{align}\label{fouR}\|T_mf\|_{L^p(\mT^2)}\leq \|\mathcal{F}_{\mR^2}^{-1}m\|_{L^1(\mR^2)}\|f\|_{L^p(\mT^2)},\quad 1\leq p\leq \infty,
\end{align}
\begin{align}\label{fouT}\|T_mf\|_{L^p(\mT^2)}\leq \|\mathcal{F}_{\mT^2}^{-1}m\|_{L^1(\mT^2)}\|f\|_{L^p(\mT^2)},\quad 1\leq p\leq \infty.
\end{align}

For $s\in\mathbb{R}$,  we set $\dot{H}^s=\{f:\int_{\mT^{2}}f \dif x=0,\  \|(-\Delta)^{s/2}f\|_{L^2}<\infty\}$ and equip it with the norm  $\|f\|_{\dot{H}^s}=(\sum_{0\neq k\in\mathbb{Z}^2}|k|^{2s}|\hat{f}(k)|^2)^{1/2}$.  We use $(\Delta_{i})_{i\geq -1}$ to denote the Littlewood--Paley blocks corresponding to a dyadic partition of unity.
Besov spaces on the torus with general indices $\alpha\in \R$, $p,q\in[1,\infty]$ are defined as
the completion of $C^\infty(\mathbb{T}^{d})$ with respect to the norm
$$
\|u\|_{B^\alpha_{p,q}}:=\left(\sum_{j\geq-1}2^{j\alpha q}\|\Delta_ju\|_{L^p}^q\right)^{1/q},$$
with the usual interpretation as the $\ell^\infty$-norm when $q=\infty.$
The H\"{o}lder--Besov space $C^\alpha$ is given by $C^\alpha=B^\alpha_{\infty,\infty}$. Let $s\in\mathbb{R}$. If $f\in B^s_{\infty,1}$ with mean zero, then $f\in \dot{H}^{s}\cap C^s$
and
\begin{align}\label{besov}\|f\|_{\dot{H}^{s}}+\|f\|_{C^{s}}\lesssim \|f\|_{B^s_{\infty,1}}.\end{align}

As in \cite{CKL21}, for two vector-valued functions $v,w :\mathbb{R}^2 \to \mathbb{R}^2 $ we denote
    $$v \approx w\quad \textrm{if }\quad v=w+\nabla^\perp p$$
    holds for some smooth scalar function $p$.
Define the projection operator for $\lambda\geq1$
$$\widehat{(P_{\leq\lambda}g)}(k)=\psi \left(\frac{k}{\lambda}\right)\hat{g}(k).$$
Here $\psi\in C_c^\infty(\mathbb{R}^2)$ satisfies $\psi(k)=0$ for $|k|\geq1$ and $\psi(k)=1$ for $|k|\leq 1/2$. We also introduce the Riesz-type transform $\mathcal{R}_j^o$, $j=1,2$ as follows
$$\hat{\mathcal{R}}_1^o(k_1,k_2)=\frac{25(k_2^2-k_1^2)}{12|k|^2},\quad \hat{\mathcal{R}}_2^o(k_1,k_2)=\frac{7(k_2^2-k_1^2)}{12|k|^2}+\frac{4k_1k_2}{|k|^2}.$$
We define the  norm
$$\|q\|_X:=\|q\|_{L^\infty}+\sum_{j=1}^2\|\mathcal{R}_j^oq\|_{L^\infty}.$$

Finally, we frequently use the  following notations. For $p\in[1,\infty)$, $\beta\in(0,1)$ and a Banach space $B$ with norm $\|\cdot\|_B$ we denote
	$$
	\$u\$^p_{B,p}:=\sup_{t\geq0}\mathbf{E}\left[\sup_{t\leq s\leq t+1}\|u(s)\|^p_{B}\right],\quad  {\$u\$^p_{C^\beta_{t}B,p}:=\sup_{t\geq0 }\mathbf{E}\left[\|u(s)\|^p_{C^\beta_tB}\right]}.
	$$
	
	\subsection{The noise}
We recall the definition of a space-time white noise $\xi$. Formally, $\xi$ is a centered Gaussian distribution with covariance
$$\mathbf{E}\xi(t,x)\xi(t',x')=\delta(t-t')\delta(x-x'),$$
for $t,t'\in\mR^+$, $x,x'\in \mT^2$, where $\delta$ denotes the Dirac delta function. 
In this paper, we consider the space-time white noise with zero spatial mean for simplicity. The space--time white noise $\xi$ with zero spatial mean is the time derivative of an $L^2$-cylindrical Wiener process $B$ with zero spatial mean on some stochastic basis $(\Omega,\mathcal{F},(\mathcal{F}_t)_{t\geq 0},\mathbf{P})$, i.e. $\xi(\phi)=\int \<\phi, \dif B\>$ for $\phi\in L^2(\mR^+\times \mT^2)$ with $\int_{\mT^2}\phi\dif x=0$. In this case $\E\xi(\phi)^2=\|\phi\|_{L^2(\mR^+\times \mT^2)}^2$.

\section{Decomposition}
\label{s:d}

We study  the
SQG equation driven by a fractional derivative of the space-time white noise, namely
\begin{align}\label{eq1}
\begin{aligned}
\partial_t \theta + u \cdummy \nabla \theta &= -\nu \Lambda^{\gamma} \theta +
 \Lambda^\alpha\xi,\\
  u=\nabla^\perp\Lambda^{-1}\theta&=(-\partial_2\Lambda^{-1}\theta,\partial_1\Lambda^{-1}\theta)=(-\mathcal{R}_2\theta, \mathcal{R}_1\theta),\\
  \theta(0)&=\theta_0,
  \end{aligned}
  \end{align}
where $\gamma\in[0,3/2)$, ${\alpha\in[0,1/4)}$ and $\xi$ is a space-time white noise  with zero spatial average on some probability space $(\Omega,\mathcal{F}, \bf{P})$. The initial condition $\theta_{0}$ is generally random, independent of the noise $\xi$, has zero mean and satisfies
for some $r\geq 1$ and $\eta>1/2$
\begin{equation}\label{eq:u0}
\mathbf{E}\|\theta_{0}\|^{2r}_{C^\eta}<\infty.
\end{equation}
We denote by  $(\cF_t)_{t\geq0}$ the  normal filtration generated by the initial data and the cylindrical
Wiener process $B$ such that $\p_tB=\xi$.
For notational simplicity we consider the unitary viscosity $\nu=1$.

The main part of our construction is inspired by the convex integration result from \cite{CKL21}, where existence and non-uniqueness of non-trivial steady state (i.e. time independent) solutions to the deterministic SQG equation with $\xi=0$ was established. The highly irregular case of \eqref{eq1} requires further decomposition and probabilistic arguments.

First, we let $v := \Lambda^{- 1} \theta$ and observe that the equation solved by $v$ reads as
\begin{equation}\label{eq:f}
\begin{aligned}
 \nabla \cdummy (-\partial_t \mathcal{R}v + \Lambda v \nabla^{\perp} v)& =
   \nabla \cdummy ( \Lambda^{\gamma } \mathcal{R} v) + \nabla \cdummy
   (-\mathcal{R} \Lambda^{- 1+\alpha } \xi),\\
    v(0)&=\Lambda^{-1}\theta_0.
    \end{aligned}
    \end{equation}
Hence,  we aim at building a convex integration procedure to solve the following equation
  \begin{equation}\label{eq:f1}
\begin{aligned}
  -\partial_t \mathcal{R}v + \Lambda v \nabla^{\perp} v& \approx
    \Lambda^{\gamma} \mathcal{R} v -\mathcal{R} \Lambda^{- 1+\alpha } \xi,\\
    v(0)&=\Lambda^{-1}\theta_0,
    \end{aligned}
    \end{equation}
    where the notation $\approx$ was explained in Section~\ref{s:not}.
 To this end, we introduce a  modification of the so-called Da Prato--Debussche trick, permitting to separate the difficulties of probabilistic and analytic nature. More precisely, we let $v:=z+f$ where
 \begin{align}\label{eqz}\partial_t z = -\Lambda^{\gamma_1}  z+ \Lambda^{- 1  {+\alpha}} \xi,\quad
    z(0)=\Lambda^{-1}\theta_0,\end{align}
 \begin{align}\label{eqv}  -\partial_t \mathcal{R}f +\Lambda^{\gamma_1}\mathcal{R}  z-\Lambda^{\gamma} \mathcal{R} z+ \Lambda (f+z) \nabla^{\perp} (f+z) \approx
    \Lambda^{\gamma} \mathcal{R} f,\quad
    f(0)=0,\end{align}
with  an  additional parameter $\gamma_{1}\in (1+2\alpha, 2-2\alpha)$.
 {The key point here is} that $\gamma_{1}$ can  generally be chosen  bigger than $\gamma$ in order to increase the regularity of $z$ so that the convex integration construction can be applied to \eqref{eqv}.

	Recall that the conventional Da Prato--Debussche technique,  {i.e. the case $\gamma_{1}=\gamma$,} is  {used to remove the most irregular part $\Lambda^{-1+\alpha}\xi$} from the equation in the subcritical regime. In the critical or supercritical regime, as the nonlinear term is worse than the noise this trick breaks down. Moreover, when $\gamma_1$ is not equal to $\gamma$, the extra term $\Lambda^{\gamma_1}z$ appears in the nonlinear equation.  This additional term formally possesses the same spatial regularity as  {$\Lambda^{-1+\alpha}\xi$}. Despite this, the temporal irregularity is eliminated, leading to an improvement in the  {time} regularity of  {$\Lambda^{\gamma_{1}}z$ compared to $\Lambda^{-1+\alpha}\xi$.}  {To summarize,} this modified version of the Da Prato--Debussche technique effectively converts the irregularities in time into irregularities in space.
 {Moreover,  we further profit from $z$ enjoying higher regularity compared to the case $\gamma_{1}=\gamma$, which is essential for the treatment of several other terms in \eqref{eqv}}.

 It will be shown in Section~\ref{s:proba} that
    $z\in CC^{\gamma_1/2 {-\alpha}-\kappa}$ for every $\kappa>0$ $\bf{P}$-a.s. Moreover, \eqref{eqv} contains an analytically ill-defined product $\Lambda z\nabla^{\perp}z$ which we will make sense of by means of probabilistic renormalization. Namely, if $\eta>\gamma_{1}/2-\alpha$ we will define the analytically ill-defined product $:\Lambda z \nabla^{\perp} z:\in C_T C^{\gamma_1 {-2\alpha}-2-\kappa}\cap C^{1-1/\gamma_1 {-2\alpha/\gamma_1}-\kappa}_TC^{-1+\kappa}$ for every $\kappa>0$. Heuristically, for  the convex integration we require both $\Lambda^{\gamma_1}z$ and $:\Lambda z \nabla^{\perp} z:$ to belong to $C^{-1+}$. Thus, we are led to further conditions $-\gamma_1/2-\alpha>-1$ and $\gamma_1-2\alpha-2>-1$, that is, $\gamma_{1}\in (1+2\alpha, (2-2\alpha)\wedge (2\eta+2\alpha))$ and $\alpha\in [0,1/4)$.

\section{Probabilistic construction}
\label{s:proba}

In this section, we study the linear equation \eqref{eqz} and perform renormalization in order to define $:\Lambda z \nabla^{\perp} z:$.

      \bp\label{fe z}
      Suppose that \eqref{eq:u0} holds for some $r\geq 1$ and $\eta>1/2$.
If $\kappa>0$ small enough then there exists $L\geq 1$ such that
  \begin{equation}\label{z:1}
  	\$z\$^{2r}_{C^{\gamma_1/2-\kappa- {\alpha}},2r}+\$z\$^{2r}_{C_{t}^{\frac{1-\alpha}{\gamma_1}-\frac12-\kappa}C^{\gamma_1-1+\kappa}, 2r}+ {\$z\$^{2r}_{C_{t}^{\frac12-\frac{\gamma+\alpha-1}{\gamma_1}-\kappa}C^{\gamma-1+\kappa},2r}}+\$z\$^{2r}_{ {C_{t}^{1/2-\kappa} C^{-\alpha-\kappa},2r}}\leq L^{2r}.
  \end{equation}
  \ep

  \begin{proof}
 This result is nowadays standard in the field of singular SPDEs. For reader's convenience we present some  details and we follow the notation of \cite[Section 9]{GP17}.
Let $W$ denote the spatial Fourier transform of the space-time white noise $\xi$. Then $W$ is a space-time white noise in Fourier space in the sense that $W$ is a complex-valued centered Gaussian process on $\mathbb{R}^{+}\times (\mathbb{Z}^{2}\setminus\{0\})$ defined by the covariance
	\begin{align*}
	&\E\left[ \int_{\mathbb{R}^{+}\times (\mathbb{Z}^{2}\setminus\{0\})} f(s,k)W(\dif s\dif k) \int_{\mathbb{R}^{+}\times (\mathbb{Z}^{2}\setminus\{0\})} g(s,k)W(\dif s\dif k)\right]
	\\&=(2\pi)^{-2} \int_{\mathbb{R}^{+}\times (\mathbb{Z}^{2}\setminus\{0\})} f(s,k) g(s,-k)\dif s\dif k,
		\end{align*}
where $(s,k)\in \mathbb{R}^{+}\times (\mathbb{Z}^{2}\setminus\{0\})$, $\dif s\dif k$ on the right hand side denotes the product of the Lebesgue measure on $\mathbb{R}^{+}$ and the counting measure on $\mathbb{Z}^{2}\setminus\{0\}$, and  $f,g$ are arbitrary  complex-valued functions in  $ L^{2}(\mathbb{R}^{+}\times (\mathbb{Z}^{2}\setminus\{0\}))$.

With this notations, we observe that the solution to \eqref{eqz} is given by the mild formulation
  	\begin{align*}
  		z(t,x)=e^{-t\Lambda^{\gamma_1}}\Lambda^{-1}\theta_0+\int_{\mR^+\times (\mZ^2\backslash\{0\})} e^{ik\cdot x}e^{-|k|^{\gamma_1}(t-s)}\1_{s\leq t}|k|^{\alpha-1}W(\dif s\dif k)=:z_1(t,x)+z_2(t,x).
  	\end{align*}
  The regularity of $z_1$ is $C([0,\infty);C^{\eta+1})\cap C_t^{1/2}C^{\eta+1-\gamma_1/2}$ for all $t\geq 0$ as a consequence of the  standard Schauder estimate (see e.g. \cite[Lemma 2.8]{ZZZ22}).
Regarding the stochastic convolution $z_{2}$, we consider the Littlewood--Paley blocks $(\Delta_{j})_{j\geq-1}$ and  have for $j\geq-1$
  	\begin{align*}
  		\E|(\Delta_j z_2)(t,x)|^2\lesssim \int_{\mR^+\times (\mZ^2\backslash\{0\})} \theta_j(k)^2e^{-2|k|^{\gamma_1}(t-s)}\1_{s\leq t}|k|^{2\alpha-2} \dif s\dif k\lesssim 2^{(2\alpha-\gamma_1)j},
  	\end{align*}
  	where  $(\theta_j)_{j\geq-1}$ is the dyadic partition of
  	unity associated to the Littlewood--Paley blocks. Therefore, in view of the  Gaussian hypercontractivity we obtain for every $p\geq1$ and $t\geq0$
  	\begin{align*}
  		\E\|z_2(t)\|_{B^{\frac{\gamma_1-\kappa}2-\alpha}_{p,p}}^p\leq (p-1)^{p/2}\tilde{L}^p,
  	\end{align*}
	for some $\tilde L>0$.
  	By Besov embedding we deduce that
  	\begin{align*}
  		\E\|z_2\|_{C^{\frac{\gamma_1}2-\alpha-\kappa}}^p\leq (p-1)^{p/2}\tilde{L}^p,
  	\end{align*}
	where $\tilde L$ may change from line to line.
 As the next step,  it is standard to consider the time increment $z_2(t)-z_2(r)$ and rewrite it in the case of  $0\leq r\leq t$ as
  	\begin{align*}
  		&z_2(t,x)-z_2(r,x)
  		\\&=\int_0^te^{-|k|^{\gamma_1}(t-s)}\int e^{ik\cdot x}|k|^{\alpha-1}W(\dif s\dif k )-\int_0^re^{-|k|^{\gamma_1}(r-s)}\int e^{ik\cdot x}|k|^{\alpha-1}W(\dif s\dif k)
  		\\&=\int_0^r(e^{-|k|^{\gamma_1}(t-s)}-e^{-|k|^{\gamma_1}(r-s)})\int e^{ik\cdot x}|k|^{\alpha-1}W(\dif s\dif k)
  		\\&\qquad+\int_r^te^{-|k|^{\gamma_1}(t-s)}\int e^{ik\cdot x}|k|^{\alpha-1}W(\dif s\dif k).
  	\end{align*}
  	When we now calculate $\E|\Delta_j z_2(t,x)-\Delta_jz_2(r,x)|^2$, we have one factor $e^{-|k|^{\gamma_1}(t-s)}-e^{-|k|^{\gamma_1}(r-s)}$ from the first term on the right hand side, which can be controlled by $e^{-|k|^{\gamma_1}(r-s)}[(|k|^{\gamma_1}|t-r|)\wedge 1]$. We also have another factor $|1-e^{-|k|^{\gamma_1}(t-r)}|$, which can also be controlled by $(|k|^{\gamma_1}|t-r|)\wedge 1$.
  	Hence, we obtain  for $\delta\in [0,1]$
  	\begin{align*}
  		\E|\Delta_j z_2(t,x)-\Delta_jz_2(r,x)|^2\lesssim|t-s|^\delta
  		2^{j(\gamma_1\delta+2\alpha-\gamma_1)},
  	\end{align*}
  	where the implicit constant is independent of $t$ and $r$.
  	Using again the Gaussian hypercontractivity together with Kolmogorov's continuity criterion \eqref{z:1} follows by choosing $\delta=1-\kappa$ as well as $$\frac2{\gamma_1}-\frac{2\alpha}{\gamma_1}-1-2\kappa<\delta<\frac2{\gamma_1}-\frac{2\alpha}{\gamma_1}-1-\frac{2\kappa}{\gamma_1},\qquad
1-\frac{2(\gamma+\alpha-1)}{\gamma_1}-2\kappa<\delta<1-\frac{2(\gamma+\alpha-1)}{\gamma_1}-\frac{2\kappa}{\gamma_1}.$$
  \end{proof}

Next, we construct the renormalized  product $:\Lambda z\nabla^{\perp}z:$. To this end, we denote by $z_{\varepsilon}$ a spatial mollification of $z$ and define the associated renormalized product as
$$
:\Lambda z_{\eps}\nabla^{\perp} z_{\eps}:=\Lambda z_{\eps} \nabla^{\perp} z_{\eps}-\E (\Lambda z_{2,\eps} \nabla^{\perp} z_{2,\eps}),
$$
where $z_{2,\eps}$ is the spatial mollification of $z_2$ defined in the proof of Proposition \ref{fe z}.
We will see in Proposition~\ref{lem:re} that by choosing a suitable mollifier the above expected value is zero by symmetry.

  \bp\label{lem:re}
  Suppose that \eqref{eq:u0} holds with $r\geq 1$ and $\eta>\frac{\gamma_1}2-\alpha$.
  There exists random distribution $:\Lambda z \nabla^{\perp} z:\in C_TC^{\gamma_1 {-2\alpha}-2-\kappa}$ $\bP$-a.s. for every $\kappa>0$, such that $:\Lambda z_\eps \nabla^{\perp} z_\eps:\to:\Lambda z \nabla^{\perp} z:$ in $C_TC^{\gamma_1 {-2\alpha}-2-\kappa}$ $\bP$-a.s. for every $\kappa>0$.
  Furthermore, if  $\kappa>0$ small enough then there exists $L\geq 1$ such that
  \begin{equation*}
\sup_{t\geq0}\mathbf{E}\left[\sup_{t\leq s\leq t+1}\|:\Lambda z \nabla^{\perp} z:(s)\|^{r}_{C^{\gamma_1 {-2\alpha}-2-\kappa}}+\|:\Lambda z \nabla^{\perp} z:\|^{r}_{C^{1-1/\gamma_1 {-2\alpha/\gamma_1}-\kappa}_tC^{-1+\kappa}}\right]\leq L^{r}.
  \end{equation*}
  \ep

  \begin{proof}
  We recall from the proof of Proposition~\ref{fe z} that $z$ splits into its initial part $z_{1}$ and the stochastic convolution $z_{2}$. Let $z_{2,\eps}$ be a space mollification of $z_2$. By a direct calculation we see that
  	\begin{align*}
  		\E(\Lambda z_{2,\eps}  \nabla^{\perp} z_{2,\eps} )=\int_{\mR\times (\mZ^2\backslash\{0\})} e^{-2|k|^{\gamma_1}(t-s)}\varphi_\eps(|k|)^2\1_{0\leq s\leq t}|k|^{2\alpha}\frac{k^{\perp}}{|k|}\dif s\dif k=0,
  	\end{align*}
because the integrand is antisymmetric under the change of variables $k\mapsto -k$. Here,  $\varphi_\eps(|k|)=\varphi(\eps|k|)$ for a smooth compactly supported function $\varphi$ with $\varphi(0)=1$. In particular, the mollification is given by a radially symmetric cut-off in Fourier space.

 Thus, the renormalized product
  \begin{align*}
  	:\Lambda z_{2,\eps}\nabla^{\perp} z_{2,\eps}:=\Lambda z_{2,\eps} \nabla^{\perp} z_{2,\eps}-\E (\Lambda z_{2,\eps} \nabla^{\perp} z_{2,\eps}),
  \end{align*}
reads as
  	\begin{align*}
  		:\Lambda z_{2,\eps}  \nabla^{\perp} z_{2,\eps} :(t,x)
&=\int e^{i(k_1+k_2)x}e^{-|k_1|^{\gamma_1}(t-s_1)}e^{-|k_2|^{\gamma_1}(t-s_2)}\1_{0\leq s_1\leq t}\1_{0\leq s_2\leq t}\\
		&\qquad\qquad\times\varphi_\eps(|k_2|)\varphi_\eps(|k_1|)|k_1|^\alpha|k_2|^\alpha\frac{k_2^{\perp}}{|k_2|}W(\dif \eta_1)W(\dif \eta_2),
  	\end{align*}
  	where $\eta_i=(s_i,k_i)\in \mR^{+}\times (\mZ^2\backslash\{0\})$ for $i=1,2$.	
  	 Similarly to  the proof of Proposition \ref{fe z} it holds
  	for $j\geq-1$
  	\begin{align*}
  		&
  		\E|\Delta_j (:\Lambda z_{2,\eps}  \nabla^{\perp} z_{2,\eps} :)(x)|^2
  		\\\lesssim& \int \theta_j(k_1+k_2)^2e^{-2|k_1|^{\gamma_1}(t-s_1)}e^{-2|k_2|^{\gamma_1}(t-s_2)}\1_{0\leq s_1,s_2\leq t}\varphi_\eps(|k_1|)^2\varphi_\eps(|k_2|)^2|k_1|^{2\alpha}|k_2|^{2\alpha}\dif \eta_1\dif \eta_2
  		\\\lesssim& \int \theta_j(k_1+k_2)^2\frac1{|k_1|^{\gamma_1-2\alpha}|k_2|^{\gamma_1-2\alpha}}\dif k_{1}\dif k_{2}\lesssim 2^{j(4-2\gamma_1+4\alpha)},
  	\end{align*}
  	where we used \cite[Lemma 3.10]{ZZ15} in the last step. Thus, using the Gaussian hypercontractivity, Besov embedding and Kolmogorov's continuity criterion and a similar argument as in Proposition~\ref{fe z}, we obtain the desired regularity for $:\Lambda z_{2,\varepsilon} \nabla^{\perp} z_{2,\varepsilon}:$ uniformly in the mollification parameter $\varepsilon$, which also guarantees the existence of a limit as $\varepsilon\to 0$. We denote this limit as $:\Lambda z_2\nabla^{\perp} z_2:$.
	
	Consequently,  we  define
 \begin{align*}
  :\Lambda z \nabla^{\perp} z:=\Lambda z_1 \nabla^{\perp} z_1+\Lambda z_1 \nabla^{\perp} z_2+\Lambda z_2 \nabla^{\perp} z_1+:\Lambda z_2\nabla^{\perp} z_2:
  \end{align*}
Since $z_1\in C([0,\infty);C^{\eta+1})\cap C^{1/2}_tC^{\eta+1-\gamma_1/2}$ with $\eta>1/2>1+\alpha-\gamma_1/2$, we have
\begin{align*}
&	\E\left[\sup_{t\leq s\leq t+1}\Big(\|\Lambda z_1 \nabla^{\perp} z_1\|_{C^{\gamma_1-2\alpha-2-\kappa}}^r+\|\Lambda z_1 \nabla^{\perp} z_2\|_{C^{\gamma_1-2\alpha-2-\kappa}}^r+\|\Lambda z_2 \nabla^{\perp} z_1\|_{C^{\gamma_1-2\alpha-2-\kappa}}^r\Big)\right]
	\\&\lesssim\E\left[\sup_{t\leq s\leq t+1}\Big(\| \theta_0 \|_{C^{\eta}}^{2r}+\| \theta_0 \|_{C^\eta}^r\| z_2\|_{C^{\gamma_1/2-\alpha-\kappa}}^r\Big)\right]\lesssim L^r.
\end{align*}
Similarly, using $\eta>\gamma_1/2-\alpha$
we have
\begin{align*}
	&	\E\Big[\|\Lambda z_1 \nabla^{\perp} z_1\|_{C_t^{1-1/\gamma_1-2\alpha/\gamma_1-\kappa}C^{-1+\kappa}}^r+\|\Lambda z_1 \nabla^{\perp} z_2\|_{C_t^{1-1/\gamma_1-2\alpha/\gamma_1-\kappa}C^{-1+\kappa}}^r
	\\&\qquad+\|\Lambda z_2 \nabla^{\perp} z_1\|_{C_t^{1-1/\gamma_1-2\alpha/\gamma_1-\kappa}C^{-1+\kappa}}^r\Big]
	\\&\lesssim\E\Big[\| \theta_0 \|_{C^{\eta}}^{2r}+\| \theta_0 \|_{C^\eta}^r\Big(\| z_2\|_{C_t^{1-1/\gamma_1-2\alpha/\gamma_1-\kappa}C^{1+\alpha-\gamma_1/2}}^r+\| z_2\|_{C_tC^{\gamma_1/2-\alpha-\kappa}}^r\Big)\Big]\lesssim L^r.
\end{align*}
 Here, we used the interpolation and Proposition \ref{fe z} to bound $\| z_2\|_{C_t^{1-1/\gamma_1-2\alpha/\gamma_1-\kappa}C^{1+\alpha-\gamma_1/2}}$ and
$\| z_1\|_{C_t^{1-1/\gamma_1-2\alpha/\gamma_1-\kappa}C^{2+2\alpha+\eta-\gamma_1}}$.
Hence, the result follows.
  \end{proof}

\begin{remark}
We observe that ignoring the initial condition, the constant $L$ in both Proposition~\ref{fe z} and Proposition~\ref{lem:re} is determined by the Gaussianity of the noise together with certain universal estimates which depend on the function spaces at hand, such as the Besov embedding.
\end{remark}

\section{Analytic construction}
\label{s:a}

In this section, we develop our iterative procedure which leads to the construction of solutions to \eqref{eqv}.
We define $M_L:=CL^2$ for a universal constant $C>0$ given below and $L\geq 1$ given as the maximum of the constants $L$ appearing in Proposition~\ref{fe z} and Proposition~\ref{lem:re}.
    We consider an increasing sequence $\{\lambda_n\}_{n\in\mathbb{N}_{0}}$ which diverges to $\infty$, and a sequence $\{r_n\}_{n\in \mathbb{N}}$  which is decreasing to $0$. Specifically, we choose $a\in\mathbb{N},$ $ b\in\mathbb{N},$ $  \beta\in (0,1)$ and let
$$\lambda_n=a^{(b^n)}, \quad r_n=M_L\lambda_0^\beta\lambda_n^{-\beta},\quad \ell_n=\lambda^{-1}_n,\quad \mu_n=\lambda_n^{\zeta},$$
where  $\beta$ will be chosen sufficiently small, $a$ sufficiently large and $\zeta$ sufficiently small.
We first assume that
$$\sum_{m\in\mathbb{N}_{0}}r^{1/2}_{m}\leq M^{1/2}_L+M^{1/2}_L\sum_{m\geq1}a^{\beta/2-mb\beta/2}{\leq}M_L^{1/2}+\frac{M_L^{1/2}}{1-a^{-b\beta/2}}\leq 3M^{1/2}_L,$$
and
$$
\sum_{m\in\mathbb{N}}\lambda_0^{\beta/2}\lambda^{-{\beta}/{4b}}_{m}\leq \lambda_0^{\beta/4}+\lambda_0^{\beta/4}\sum_{m\geq1}a^{\beta/4-mb\beta/4}\leq\lambda_0^{\beta/4}+\frac{\lambda_0^{\beta/4}}{1-a^{-b\beta/4}}\leq 3\lambda_0^{\beta/4},$$
which requires
\begin{align}\label{para}a^{b\beta}\geq 16.
\end{align}
Further necessary conditions on these parameters will be given below, in particular in Section~\ref{s:p}.

The iteration is indexed by a parameter $n\in\mathbb{N}_0$. At each step $n$, a pair $(f_{  n},q_n)$ is constructed and it solves the approximate equation
\begin{align}\label{induction ps}
\begin{aligned}
-\partial_t \mathcal{R}f_{  n} + \Lambda   f_n
   \nabla^{\perp}   f_n   +\Lambda z_n \nabla^{\perp}   f_n  &+\Lambda   f_n   \nabla^{\perp} z_n\\+
   P_{\leq \lambda_n}(:\Lambda z \nabla^{\perp} z:+\Lambda^{\gamma_1}\mathcal{R}  z-\Lambda^{\gamma}\mathcal{R}  z) &\approx \Lambda^{\gamma } \mathcal{R}   f_n   +
   \nabla q_n,\\
     f_n  (0)&=0,
   \end{aligned}
   \end{align}
   where $z_n=P_{\leq \lambda_n} z$ with $P_{\leq\lambda_{n}}$ introduced in Section~\ref{s:not}.
In comparison to \eqref{eqv}, this equation includes the additional error $ q_{n}$, which shall eventually vanish with $n\to\infty$, as well as the additional projection $P_{\leq \lambda_n}$. Both $  f_n  $ and $ q_{n}$ are compactly supported in Fourier on frequencies of order at most $\lambda_{n}$. We observe that the noise is added in \eqref{induction ps} scale by scale. This is needed in order to preserve the frequency localization.
Applying the  divergence to \eqref{induction ps} yields therefore  an approximation of \eqref{eq:f}. Due to the application of divergence, we only need to consider the mean-free parts of all the terms in \eqref{induction ps}. For notational simplicity, we do so without further modifying the notation.

As the next step, we start the iteration by letting $f_{  0}=0$.
Then,  \eqref{induction ps} yields
  $$\nabla q_0 \approx
   P_{\leq \lambda_0}(:\Lambda z
   \nabla^{\perp} z: +\Lambda^{\gamma_1}z-\Lambda^\gamma z) ,$$
  and taking $q_{0}$ mean-free we estimate for $\kappa>0$
  $$\|q_0\|_X\lesssim \|z\|_{C^{\gamma_1/2-\kappa {-\alpha}}}+\|:\Lambda z
   \nabla^{\perp} z:\|_{C^{\gamma_1-2-\kappa {-2\alpha}}}.$$
Here, we used $\gamma_1-2\alpha-2>-1$, $-\frac{\gamma_1}{2}-\alpha>-1$, $\gamma_1/2-\gamma-\alpha>-1$.
   Thus,  Proposition \ref{fe z} and Proposition~\ref{lem:re} imply
   \begin{align*}
\left(\sup_{t\geq 0}\mathbf{E}\sup_{s\in[t,t+1]}\|q_0(s)\|^r_{X}\right)^{1/r}\leq CL^2\leq
M_L.
\end{align*}

Given $  f_0  ,\dots   f_n  $, the next iteration  $ f_{  n+1}     $ is constructed through an induction argument so that the following estimates  hold true for every $n\in\mathbb{N}_{0}$
 \begin{equation}\label{iteration ps}
\left(\sup_{t\geq 2^{-n+2}}\mathbf{E}\sup_{s\in[t,t+1]}\| f_{  n+1}     (s)-  f_n  (s)\|^{2r}_{B^{1/2}_{\infty,1}}\right)^{1/2r}\leq
M_0 r_{n}^{1/2}+r^{1/2}_{n+1},
\end{equation}
 \begin{equation}\label{iteration ps1}
\left(\mathbf{E}\sup_{s\in[2^{-n-1},2^{-n+2}]}\| f_{  n+1}     (s)-  f_n  (s)\|^{2r}_{B^{1/2}_{\infty,1}}\right)^{1/2r}\leq
M_0 M_L^{1/2}+r^{1/2}_{n+1},
\end{equation}
 \begin{equation}\label{iteration ps2}
\left(\mathbf{E}\sup_{s\in[0,2^{-n-1}]}\| f_{  n+1}     (s)-  f_n  (s)\|^{2r}_{B^{1/2}_{\infty,1}}\right)^{1/2r}\leq
r^{1/2}_{n+1},
\end{equation}
and for $1/2<\vartheta<1/2+\beta/4b$
\begin{align}\label{induction fn}\left(\sup_{t\geq 2^{-n+2}}\mathbf{E}\sup_{s\in[t,t+1]}\| f_{  n+1}     (s)-  f_n  (s)\|^{2r}_{C^{\vartheta}}\right)^{1/2r}\leq M_0M_L^{1/2}\lambda_0^{\beta/2}\lambda^{-{\beta}/{4b}}_{n+1}+r^{1/2}_{n+1},
\end{align}
and for $\delta>\beta/2$
\begin{align}\label{induction w}
\$ f_{  n+1}     -  f_n  \$_{B^{1/2-\delta}_{\infty,1},2r}\leq r^{1/2}_{n+1},
\end{align}
and
\begin{align}\label{iteration R}
\left(\sup_{t\geq 2^{-n}}\mathbf{E}\sup_{s\in[t,t+1]}\|q_{n+1}(s)\|^r_{X}\right)^{1/r}\leq
r_{n+1},
\end{align}
\begin{align}\label{iteration R1}
\left(\mathbf{E}\sup_{s\in[0,2^{-n}]}\|q_{n+1}(s)\|^r_{X}\right)^{1/r}\leq M_L+\sum_{k=1}^{n+1}r_{k}\leq 3M_L,
\end{align}
and
\begin{equation}\label{inductionv ps}
\$ f_{  n+1}     (t)\$_{B^{1/2}_{\infty,1},2r}\leq10M_0M_L^{1/2}
\end{equation}
for a constant $M_0$ independent of $a,b,\beta,M_L$, and
\begin{align}\label{eq:C1}
\$ f_{  n+1}     \$_{C^{1/2}_tB^{1/2}_{\infty,1},2r}\leq M_0M^{1/2}_L\lambda^{1/2}_{n+1},
\end{align}
and for $\delta>\beta/2$
\begin{align}\label{induction wf}\| f_{  n+1}     \|_{C_{t}^{1/2}B^{-\delta}_{\infty,1},2r}\leq M_L^{1/2}+M_0M^{1/2}_L\sum_{k=1}^{n+1}\lambda_{k}^{-\delta}\leq M_L^{1/2}+3M_0M_L^{1/2}.\end{align}

\begin{proposition}\label{p:iteration}
Assume \eqref{eq:u0} holds true for some $r\geq1$ and $\eta>1/2$. There exists a choice of parameters $a, b, \beta,\zeta$ such that the following holds true: Let $( f_i   ,q_i)_{i=0}^n$ for some $n\in\N_{0}$ be  $(\mathcal{F}_t)$-adapted so that
\begin{enumerate}
\item for every $i=\{0,\dots,n\}$, $( f_i   ,q_i)$ solves \eqref{induction ps} with $n$ replaced by $i$,
\item for every $i=\{0,\dots,n\}$ and every $t\geq0$, the frequencies of $ f_i   (t)$ and $q_i(t)$ are localized in a ball of radius $\leq 6\lambda_i$ and $\leq 12\lambda_i$, respectively,
\item the estimates \eqref{iteration R}-\eqref{induction wf} hold true with $n+1$ replaced by $i\in \{0,\dots,n\}$,
\item the estimates  \eqref{iteration ps}-\eqref{induction w} hold true with $n$ replaced by $i\in \{0,\dots,n-1\}$.
\end{enumerate}
 Then    there exists an $(\mathcal{F}_t)$-adapted  $( f_{  n+1}     ,q_{n+1})$ which solves \eqref{induction ps} on the level $n+1$ and such that for every $t\geq 0$ the frequencies of $ f_{  n+1}     (t)$ and $q_{n+1}(t)$ are localized in a ball of radius $\leq 6\lambda_{n+1}$ and $\leq 12\lambda_{n+1}$, respectively, and the estimates
 \eqref{iteration ps}-\eqref{induction wf} are satisfied.

Moreover, for any $\eps>0$ we can choose the parameter $a$ large enough depending on $M_L$ such that
\begin{align}\label{induction w1}\$ f_{  n+1}     -  f_n  \$_{B^{1/2-\delta}_{\infty,1},2r}\leq \eps/2^{n+1}.
\end{align}
\end{proposition}

The detailed proof of this key technical result is presented in Section~\ref{s:it}. Nevertheless, in order to facilitate the first  reading of the manuscript, the reader may assume that it has already been proven and directly continue with Section~\ref{s:1.1} and Section~\ref{s:in}. Here, we apply Proposition~\ref{p:iteration}  to establish our main results concerning the solutions to the initial value problem as well as stationary and ergodic stationary solutions.

\section{Detailed iteration -- proof of Proposition~\ref{p:iteration}}
\label{s:it}

This section is devoted to the proof of Proposition~\ref{p:iteration}, split into several subsections. In the first step, we summarize the conditions on our parameters  needed to recover all the iterative assumptions. After a preliminary step of mollification, we introduce the perturbation $g_{n+1}$ which in turn  permits to define the next iteration $ f_{  n+1}     $. The inductive estimates on $ f_{  n+1}     $ are proved in Section~\ref{s:8.4}. The new error $q_{n+1}$ is implicitly defined via the mollification of  \eqref{induction ps} on the level $n$ and \eqref{induction ps} on the level $n+1$. The  inductive estimates on $q_{n+1}$ are established in Section~\ref{s:estq}. While some estimates are similar to our previous work \cite{HZZ22} or to the deterministic result \cite{CKL21}, major difficulties arise in the treatment of certain terms  including the linear solution $z$ due to the supercritical nature of our setting. These estimates are very delicate and are among  the main novelties of this manuscript. See in particular the control of $q_{z}$ in Section~\ref{s:estq}.

\subsection{Conditions on parameters}\label{s:p}

In the sequel, we use the following basic bounds on the parameters
$$\beta<3/2-\gamma,\quad \beta+\frac{3}{2b}+\kappa<1/2,\quad \zeta>1/b+\beta/2.$$
The following conditions are necessary for the terms containing $z$ (see Section~\ref{s:estq} below): for the control of  $q_N$, we postulate
$$ (1-\frac{\gamma_1}{2} {-\alpha})\wedge (\gamma_1-1-2\alpha)\wedge (1+\frac{\gamma_1}{2}-\gamma-\alpha)>b \beta+2\kappa,$$
$$(1-\frac1{\gamma_1}-\frac{2\alpha}{\gamma_1})\wedge (\frac{1}{\gamma_1} -\frac{\alpha}{\gamma_1}-\frac12)\wedge (\frac12-\frac{\gamma+\alpha-1}{\gamma_1})>\beta+\kappa,$$
whereas for the control of  $q_z$ we require
$$\zeta+\beta b+3\kappa<\frac{\gamma_1}2-{\alpha}-\frac12.$$
We can choose  $ {2(\gamma+\alpha-1)<1+2\alpha<\gamma_1<(2-2\alpha)\wedge (2\eta+2\alpha)}$ and  $\zeta>0$ small enough such that
$$\zeta<\frac{\gamma_1}2-{\alpha}-\frac12.$$  Then we choose $b$ large enough such that $\zeta>1/b$ and all the conditions  are satisfied if $\beta$ and $\kappa$ are small enough. Moreover, we can choose $a$ large enough to absorb implicit constant  and such that \eqref{para} holds.

\subsection{Mollification step}

As usual in the number of convex integration schemes, we start with a mollification of the previous iteration $  f_n  $. To this end, we consider  $\{\varphi_\varepsilon\}_{\varepsilon>0}$, a family of standard positive mollifiers with support of $\varphi$ in $(0,1)$. We use one-sided mollifiers as this will  preserve the initial datum as well as  adaptedness. Next, we  extend $  f_n   $ and $ q_n$ to negative times by taking the values at $t=0$. Consequently, the equation also holds for $t<0$ as $\partial_t   f_n  (0)=0$ by our construction. We define the mollification of $  f_n  $, $q_n$  in  time by convolution as follows
	$$f_{\ell}=  f_n  *_t\varphi_{\ell_{n+1}},\qquad
	q_\ell=q_n*_t\varphi_{\ell_{n+1}},$$
	where  $\ell:=\ell_{n+1}=\lambda_{n+1}^{-1}$.
Then both  $f_\ell$ and $q_\ell$ are $(\mathcal{F}_t)$-adapted.
We apply mollification on  both sides of \eqref{induction ps} and obtain
\begin{align}\label{mol}
\begin{aligned}
- \partial_t \mathcal{R}f_{\ell} + \Lambda f_{\ell} \nabla^{\perp}
   f_{\ell}+ \Lambda z_n\nabla^{\perp}
   f_{\ell} + \Lambda f_{\ell} \nabla^{\perp}
  z_n +&P_{\leq\lambda_{n}} [:(\Lambda z \nabla^{\perp}
  z):+\Lambda^{\gamma_1} \mathcal{R}z-\Lambda^{\gamma} \mathcal{R} z]_\ell
   \\&\approx  \Lambda^{\gamma } \mathcal{R} f_{\ell} + \nabla q_{\ell} +
   R_{\tmop{com}} ,\\
   f_\ell(0)&=0,
   \end{aligned}
   \end{align}
where $$R_{\tmop{com}}=\Lambda f_{\ell} \nabla^{\perp}
   f_{\ell}-(\Lambda   f_n  \nabla^\bot   f_n  )*_t\varphi_{\ell}+\Lambda z_n\nabla^{\perp}
   f_{\ell} + \Lambda f_{\ell} \nabla^{\perp}
  z_n-(\Lambda z_n\nabla^{\perp}
     f_n   + \Lambda   f_n   \nabla^{\perp}
  z_n )*_t\varphi_{\ell} .$$
  In view of \eqref{eq:C1} for $f_n$ we obtain for $b(1/2-\beta/2)>\vartheta>1/2$
 \begin{align}\label{mol1}
  \begin{aligned}
\$f_{\ell}-  f_n  \$_{C^{\vartheta},2r}&\lesssim \ell^{1/2}\lambda_n^{\vartheta-1/2}\$  f_n  \$_{C^{1/2}_{t}B_{\infty,1}^{1/2},2r}
  \\&\lesssim M_0{\lambda_{n+1}^{-1/2} M^{1/2}_L\lambda^{\vartheta}_n}\leq \frac14r^{1/2}_{n+1},
  \end{aligned}
  \end{align}
  and
  \begin{align}\label{mol2}
\$f_\ell\$_{C^{1/2}_tB_{\infty,1}^{1/2},2r}\leq \$  f_n  \$_{C^{1/2}_{t}B_{\infty,1}^{1/2},2r}\leq M_0M^{1/2}_L\lambda^{1/2}_n,
  \end{align}
and
for $\delta>\beta/2$ by \eqref{induction wf} for $f_n$
\begin{align}
\label{mol3}
\$f_{\ell}\$_{C^{1/2}_{t}B^{-\delta}_{\infty,1},2r}\leq\$  f_n  \$_{C^{1/2}_{t}B^{-\delta}_{\infty,1},2r}\leq M_L^{1/2}+M_0M^{1/2}_L\sum_{k=1}^n\lambda_{k}^{-\delta}.
\end{align}

\subsection{Perturbation step -- construction of $ f_{  n+1}     $}
\label{sec:conf}

The construction of $ f_{  n+1}     $ is done by perturbing the mollified solution $f_\ell $.
We first introduce a smooth  function
\begin{align*}
\tilde{\chi}(t)=\varphi_\ell*\|q_n\|_X(t)+r_{n}.
\end{align*}
Then $\|q_\ell\|_X\leq \tilde{\chi}$ and
$\|\tilde{\chi}'(t)\|_{C^{0}}\lesssim \ell^{-1}\sup_{s\in[t-\ell,t]}\|q_n(s)\|_X$. We also introduce the following cut-off function
\begin{align*}
	\chi(t)=\begin{cases}
		0,& t\leq 2^{-n-1},\\
		\in (0,1),& t\in (2^{-n-1},2^{-n} ),\\
		1,&t\geq 2^{-n}.
	\end{cases}
\end{align*}
{Here in the middle interval we smoothly interpolate so that} it holds
$\|\chi'\|_{C^{0}}\leq 2^{n+1}$. Now, define the perturbation $g_{n+1}$ by
\begin{equation}\label{realg}
	g_{n+1}=\chi\tilde{g}_{n+1},
\end{equation}
where similarly to \cite[(2.10)]{CKL21} we define $\tilde{g}_{n+1}$ by
\begin{equation}\label{tildeg}
\tilde{g}_{n+1}(t,x)=\sum_{j=1}^2a_{j,n+1}(t,x)\cos (5\lambda_{n+1}l_j\cdot x), \quad a_{j,n+1}=2\sqrt{\frac{\tilde{\chi}}{5\lambda_{n+1}}}P_{\leq \mu_{n+1}}\sqrt{C_0+\mathcal{R}_j^o\frac{q_{\ell}}{\tilde{\chi}}},
\end{equation}
where $l_1=(\frac35,\frac45)^T$, $l_2=(1,0)^T$.
\begin{remark}\label{r:1}
A few remarks concerning the perturbation $g_{n+1}$:
	\begin{itemize}
	\item $ {g}_{n+1} $ is mean zero due to \eqref{tildeg}. As $L$  as well as the parameters $a,b,\beta$  are deterministic and $q_\ell$ is $(\mathcal{F}_t)$-adapted,   $\tilde{g}_{n+1}$ is $(\mathcal{F}_t)$-adapted.
\item We note that since $\mu_{n+1}$ is much smaller than $\lambda_{n+1}$, the spatial frequencies of $\tilde{g}_{n+1}$ are localized to $\lambda_{n+1}$. Its Fourier coefficients take the form for $k\in \mathbb{Z}^{2}$
\[
\begin{aligned}
	\widehat{\tilde{g}_{n + 1}} (k)
	&=\begin{cases}
		\sum_{j = 1}^2 \sqrt{\frac{\tilde{\chi}}{5 \lambda_{n + 1}}} \left\langle
		\sqrt{C_0 +\mathcal{R}_j^o \frac{q_{\ell}}{\tilde{\chi}}}, P_{\leq\mu_{n+1}}e^{-i (k+5 \lambda_{n + 1}
			l_j ) \cdummy x} \right\rangle , &\text{if } | k+5 \lambda_{n + 1} l_j  | \leqslant \mu_{n + 1}, \\
		\sum_{j = 1}^2\sqrt{\frac{\tilde{\chi}}{5 \lambda_{n + 1}}} \left\langle
		\sqrt{C_0 +\mathcal{R}_j^o \frac{q_{\ell}}{\tilde{\chi}}}, P_{\leq\mu_{n+1}} e^{-i (k - 5 \lambda_{n +
				1} l_j) \cdummy x} \right\rangle , &\text{if } | k - 5 \lambda_{n + 1} l_j | \leqslant \mu_{n + 1},\\
		0, &\text{otherwise}.
	\end{cases}
\end{aligned}
\]
This will be employed frequently throughout the paper.
\end{itemize}
\end{remark}
And finally,  the new iteration $ f_{  n+1}     $ is defined as
$$ f_{  n+1}     :=f_\ell+g_{n+1},$$
which is also mean zero and $(\mathcal{F}_t)$-adapted.

Next, we derive a few estimates that facilitate the verification of the inductive estimates below  in Section~\ref{s:8.4} and Section~\ref{s:estq}. In view of the definition of $\tilde{g}_{n+1}$, we obtain
\begin{align}\label{gn}
\|\tilde{g}_{n+1}(t)\|_{B_{\infty,1}^{1/2}}&= \sum_j2^{j/2}\|\Delta_j\tilde{g}_{n+1}\|_{L^\infty}\lesssim\sum_{j:2^j\thicksim \lambda_{n+1}}2^{j/2}\|\Delta_j\tilde{g}_{n+1}\|_{L^\infty}\lesssim \lambda^{1/2}_{n+1}\|\tilde g_{n+1}\|_{L^{\infty}}\nonumber
\\&\lesssim \lambda_{n+1}^{1/2}\left((C_0+1)\frac{\tilde\chi(t)}{5\lambda_{n+1}}\right)^{1/2}\leq \frac{M_0}3\tilde\chi(t)^{1/2}.
\end{align}
{Here we used the following argument in the third step: By $2^j\sim \lambda_{n+1}$ there exists $c_{1},c_{2}>0$ such that $c_{1}\lambda_{n+1}\leq 2^j\leq c_{2}\lambda_{n+1}$. Hence $\log_{2} c_{1}+\log_{2} \lambda_{n+1}\leq j\leq\log_{2} c_{2}+\log_{2} \lambda_{n+1}$ and $$\sum_{j:2^j\sim \lambda_{n+1}}1\lesssim 1.$$}
For the $C^\vartheta$-estimate we have
\begin{align}
\|\tilde{g}_{n+1}(t)\|_{C^\vartheta}
&= \sup_j2^{j\vartheta}\|\Delta_j\tilde{g}_{n+1}(t)\|_{L^\infty}\lesssim\sup_{j:2^j\thicksim \lambda_{n+1}}2^{j\vartheta}\|\Delta_j  \tilde{g}  _{n+1}(t)\|_{L^\infty}\lesssim \lambda^{\vartheta}_{n+1}\|  \tilde{g}_{n+1}  (t)\|_{L^{\infty}}\nonumber
\no\\
&\lesssim \lambda_{n+1}^{\vartheta}\left((C_0+1)\frac{\tilde\chi(t)}{5\lambda_{n+1}}\right)^{1/2}\leq M_0\lambda_{n+1}^{\vartheta-1/2}\tilde\chi(t)^{1/2}.\label{est:fn1}
\end{align}
Furthermore,  we obtain for $\delta>\beta/2$
\begin{align}\label{gn3}
\|  \tilde{g}  _{n+1}(t)\|_{B_{\infty,1}^{1/2-\delta}}&= \sum_j2^{j(1/2-\delta)}\|\Delta_j  \tilde{g}  _{n+1}(t)\|_{L^\infty}
\\&\lesssim \lambda_{n+1}^{1/2-\delta}\left((C_0+1)\frac{\tilde\chi(t)}{5\lambda_{n+1}}\right)^{1/2}.\nonumber
\end{align}
Moreover, we obtain for $s\leq t\leq s+1$
\begin{align*}
&\|a_{j,n+1}(t)-a_{j,n+1}(s)\|_{L^\infty}
\lesssim \lambda_{n+1}^{-1/2}\sqrt{|\tilde{\chi}(t)-\tilde{\chi}(s)|}+\lambda_{n+1}^{-1/2}\tilde{\chi}(s)^{1/2}\sqrt{\Big\|\frac{\mathcal{R}_j^oq_\ell(t)}{\tilde{\chi}(t)}-\frac{\mathcal{R}_j^oq_\ell(s)}{\tilde{\chi}(s)}\Big\|_{L^\infty}}\nonumber
\\&\lesssim\lambda_{n+1}^{-1/2}\ell^{-1/2}\|q_n\|^{1/2}_{C_{[t-2,t]}X}|t-s|^{1/2}+\lambda_{n+1}^{-1/2}\sqrt{\frac{\|\tilde{\chi}(s)\mathcal{R}_j^oq_\ell(t)-\tilde{\chi}(t)\mathcal{R}_j^oq_\ell(s)\|_{L^\infty}}{\tilde{\chi}(t)}}\nonumber
\\&\lesssim\lambda_{n+1}^{-1/2}\ell^{-1/2}\|q_n\|^{1/2}_{C_{[t-2,t]}X}|t-s|^{1/2},
\end{align*}
which implies that
\begin{align}\label{esta}\|a_{j,n+1}\|_{C^{1/2}_tL^\infty}\lesssim \|q_n\|_{C_{[t-2,t]}X}^{1/2}.\end{align}
Finally, we obtain
\begin{align}\label{esta1}
	\|a'_{j,n+1}(t)\|_{L^\infty}&\lesssim  \lambda_{n+1}^{-1/2}\tilde{\chi}^{-1/2}|\tilde{\chi}'|+\lambda_{n+1}^{-1/2}\tilde{\chi}^{1/2}(\ell^{-1}\tilde\chi^{-1}\|q_n\|_{C_tX}+\|q_\ell\|_X|\tilde{\chi}'|\tilde{\chi}^{-2})\nonumber
	\\&\lesssim \lambda_{n+1}^{-1/2}r_{n}^{-1/2}\ell^{-1}\|q_n\|_{C_tX}.
\end{align}

\subsection{Inductive estimates for $ f_{  n+1}     $}
\label{s:8.4}

From the construction we see that $\supp\widehat{ f_{  n+1}     }\subset \{|k|\leq 6\lambda_{n+1}\}$ and $ f_{  n+1}     (0)=0$. We first prove \eqref{iteration ps} and \eqref{induction fn}. Combining  \eqref{mol1}, \eqref{gn}  we obtain
\begin{align*}&\sup_{s\geq 2^{-n+2}}\left(\mathbf{E}\sup_{t\in[s,s+1]}\| f_{  n+1}     (t)-  f_n  (t)\|^{2r}_{B_{\infty,1}^{1/2}}\right)^{1/2r}
\\&\leq \$f_\ell-  f_n  \$_{B_{\infty,1}^{1/2},2r}+\sup_{s\geq 2^{-n+2}}\left(\mathbf{E}\sup_{t\in[s,s+1]}\|g_{n+1}(t)\|^{2r}_{B_{\infty,1}^{1/2}}\right)^{1/2r}
\\&\leq r^{1/2}_{n+1}+\frac{M_0}3r^{1/2}_n+ \frac{M_0}3\sup_{s\geq 2^{-n+1}}\left(\mathbf{E}\sup_{t\in[s,s+1]} \|q_n(t)\|_X^r\right)^{1/2r}
\\&\leq r^{1/2}_{n+1}+ M_0r_n^{1/2},\end{align*}
which implies  \eqref{iteration ps}. Here in the last second step we used $\ell\leq 2^{-n+1}$ and in the last step we used \eqref{iteration R} for $q_n$.
Similarly using  \eqref{est:fn1} and $r_n=M_L\lambda_0^\beta \lambda_n^{-\beta}$  we have for $\vartheta<\frac12+\frac{\beta}{4b}$
$$\sup_{s\geq 2^{-n+2}}\left(\mathbf{E}\sup_{t\in[s,s+1]}\| f_{  n+1}     (t)-  f_n  (t)\|_{C^{\vartheta}}^{2r}\right)^{1/2r}\leq M_0M_L^{1/2}\lambda_{0}^{\beta/2}\lambda^{-\beta/4b}_{n+1}+r^{1/2}_{n+1},$$
which implies \eqref{induction fn}.
Also combining  \eqref{mol1} and \eqref{gn}, \eqref{iteration R1} for $q_n$ we obtain
$$\left(\mathbf{E}\sup_{t\in[2^{-n-1},2^{-n+2}]}\| f_{  n+1}     (t)-  f_n  (t)\|_{B_{\infty,1}^{1/2}}^{2r}\right)^{1/2r}\leq M_0M_L^{1/2}+r^{1/2}_{n+1}.$$
Using \eqref{mol1}  we obtain
$$\left(\mathbf{E}\sup_{t\in[0,2^{-n-1}]}\| f_{  n+1}     (t)-  f_n  (t)\|_{B_{\infty,1}^{1/2}}^{2r}\right)^{1/2r}\leq r^{1/2}_{n+1}.$$%
 Thus \eqref{iteration ps1} and \eqref{iteration ps2} hold. Moreover, we have for $t\geq0$
\begin{align*}
&\sup_{s\geq 0}\left(\mathbf{E}\sup_{t\in[s,s+1]}\| f_{  n+1}     (t)\|_{B_{\infty,1}^{1/2}}^{2r}\right)^{1/2r}
\\&\leq \sum_{k=0}^n\sup_{s\geq 0}\left(\mathbf{E}\sup_{t\in[s,s+1]}\|f_{  k+1}(t)-f_{  k}(t)\|_{B_{\infty,1}^{1/2}}^{2r}1_{t\in (2^{-k-1},2^{-k+2}]^c}\right)^{1/2r}
\\&\quad+\sup_{s\geq 0}\left(\mathbf{E}\sup_{t\in[s,s+1]}\left(\sum_{k=0}^n\|f_{  k+1}(t)-f_{  k}(t)\|_{B_{\infty,1}^{1/2}}1_{t\in (2^{-k-1},2^{-k+2}]}\right)^{2r}\right)^{1/2r}.
\end{align*}
The first term on the right hand side is  bounded by $\sum_{k=0}^n(M_0r_k^{1/2}+2r^{1/2}_{k+1})$ as a consequence of the division into  $s\geq 2^{-k+2}$ and $s\leq 2^{-k+2}$ and application of \eqref{iteration ps} and \eqref{iteration ps2} for $k=0,\dots,n$. For the second term we use the fact that each $t$ only belongs to three intervals of $(2^{-k-1},2^{-k+2}]$ and we apply \eqref{iteration ps1} to bound it by $3M_0M_L^{1/2}+\sum_{k=0}^nr^{1/2}_{k+1}$. Thus
by \eqref{para}, \eqref{inductionv ps} follows for $ f_{  n+1}     $.

The estimates \eqref{induction w} and \eqref{induction w1} follow from \eqref{gn3} and \eqref{mol1}. Now, we estimate the $C^{1/2}_tB^{1/2}_{\infty,1}$-norm.
 We use \eqref{esta} and \eqref{iteration R}, \eqref{iteration R1} to obtain
\begin{align}\label{fn2}
	\$  \tilde{g}  _{n+1}\$_{C_t^{1/2}B_{\infty,1}^{1/2},2r}&\leq \sum_{j:2^j\sim \lambda_{n+1}}2^{j/2}\$\Delta_j  \tilde{g}  _{n+1}\$_{C_t^{1/2}L^\infty,2r}
	\\&\lesssim \lambda^{1/2}_{n+1}\$q_n\$^{1/2}_{X,r}\leq  (M_0-2)M^{1/2}_L\lambda^{1/2}_{n+1}.\nonumber
\end{align}
The time derivative of $\chi$ is  bounded by $2^{n+1}\leq\lambda_{n+1}$. Hence  \eqref{eq:C1} holds at the level $n+1$ by \eqref{mol2}. Similarly we have
$$\$\chi  \tilde{g}  _{n+1}\$_{C^{1/2}_{t}B_{\infty,1}^{-\delta},2r}\leq M_0M_L^{1/2}\lambda_{n+1}^{-\delta},$$
Thus \eqref{induction wf} at the level  $n+1$ follows from  \eqref{mol3}.

\subsection{Inductive estimate for $q_{n+1}$}\label{s:estq}

In order to obtain a formula for the error  $q_{n+1}$ we subtract \eqref{mol} from \eqref{induction ps} at level $n+1$. This leads to
\begin{align*}
\nabla q_{n+1}&\approx \underbrace{\chi^2\Lambda   \tilde{g}  _{n+1}\nabla^\bot   \tilde{g}  _{n+1}+\nabla q_\ell}_{\nabla q_M}+\underbrace{\Lambda g_{n+1}\nabla^\bot f_{\ell}+\Lambda f_{\ell}\nabla^\bot g_{ n+1}}_{\nabla q_T}\underbrace{-  \Lambda^{\gamma}\mathcal{R}g_{n+1}}_{\nabla q_D}\underbrace{-\partial_t\mathcal{R}g_{n+1}}_{\nabla q_I}
	\\&\quad+R_{\rm{com}}+\underbrace{P_{\leq\lambda_{n+1}}(:\Lambda z\nabla^\bot z:+\Lambda^{\gamma_1} \mathcal{R} z-\Lambda^{\gamma} \mathcal{R}z) - P_{\leq\lambda_{n}}(:\Lambda z\nabla^\bot z:+\Lambda^{\gamma_1}\mathcal{R} z-\Lambda^{\gamma} \mathcal{R} z)*\varphi_\ell }_{\nabla q_N}
\\&\quad+\underbrace{\Lambda g_{n+1}\nabla^{\bot} z_{n+1} +\Lambda f_\ell\nabla^{\bot}(z_{n+1}-z_n)+\Lambda z_{n+1}\nabla^{\bot} g_{n+1} +\Lambda (z_{n+1}-z_n)\nabla^{\bot}f_\ell}_{\nabla q_z}.
\end{align*}
The main new difficulties  compared to our previous work \cite{HZZ22} appear in the part $q_{z}$. Note  that as the left hand side is mean free, the right hand side is also mean free. Then for each term on the right hand side we can subtract its mean part and in the following we do not change the notation for simplicity. For the mean free part we define the inverse of $\nabla$ by $\Delta^{-1}\nabla\cdot$.

 As a consequence of the above  formula for $q_{n+1}$, we deduce that  $\supp\hat{q}_{n+1}\subset\{|k|\leq 12\lambda_{n+1}\}$ as required in Proposition~\ref{p:iteration}. In the remainder of this section, we are concerned with the iterative estimates \eqref{iteration R} and \eqref{iteration R1}. Recall that in $q_{M}$, the perturbation $\chi\tilde g_{n+1}$ plugged in the nonlinearity cancels part of the old error $\chi^{2} q_{\ell}$, but it  only acts for  $t\geq 2^{-n-1}$. Hence, the error $q_{M}$ can only be shown to be small for $t\geq 2^{-n}$. For $t<2^{-n}$ we can only use the trivial estimate and hence $q_{M}$ is not smaller than $q_{n}$ here. Other terms, for instance the one including the time derivative of $\chi$ already see the problem for times $t\in[2^{-n-1},2^{-n}]$, which in turn dictates the formulation of \eqref{iteration R} and \eqref{iteration R1}.

   We first focus on the most difficult part $q_z$ which requires further decomposition in Fourier space and then we proceed with the remaining terms.

\subsubsection{Estimate of $q_z$}\label{sss:qz}
It will be seen below that no splitting of the time interval is needed for this term, it can be directly controlled for all $t\geq 0$.
Since we only need to consider the mean free part of $q_{z}$, it is sufficient to bound
its $C^{-1+}$-norm, meaning, the norm in $C^{-1+\varsigma}$ for some $\varsigma>0$ arbitrarily small. For $\Lambda   \tilde{g}  _{n+1}\nabla^{\bot} z_{n+1}$ which is localized at frequencies proportional to $\lambda_{n+1}$ we have by \eqref{gn} and Proposition \ref{fe z} for $\kappa>0$
\begin{align*}
\$\Lambda   g  _{n+1}\nabla^{\bot} z_{n+1}\$_{{C^{-1+}},r}&\lesssim \lambda_{n+1}^{-1}\$\Lambda   \tilde{g}  _{n+1}\nabla^{\bot} z_{n+1}\$_{C^\kappa,r}\lesssim\lambda_{n+1}^{-1}\$\Lambda   \tilde{g}  _{n+1}\$_{C^\kappa,2r}\$\nabla^{\bot} z_{n+1}\$_{C^\kappa,2r}\\
&\lesssim \lambda_{n+1}^{-1}\lambda_{n+1}^{1/2+\kappa}\lambda_{n+1}^{1-\gamma_1/2 {+\alpha}+2\kappa}M_L
\leq\frac1{35}r_{n+1}.
\end{align*}
Here we used H\"{o}lder's inequality in the second step and we used $\beta<\frac{\gamma_1}2{-\alpha}-\frac12-3\kappa$. The estimate for $\Lambda z_{n+1}\nabla^{\bot} g_{n+1}$ is similar.

The remainder of this Section~\ref{sss:qz} is devoted to the estimate of the difficult term $\Lambda f_\ell\nabla^{\bot}(z_{n+1}-z_n) +\Lambda (z_{n+1}-z_n)\nabla^{\bot}f_\ell$. Here, the smallness must come from the difference $z_{n+1}-z_{n}$ but it is delicate to capture without postulating further restrictions on our parameters. In the first step, we recall that $\ell=\ell_{n+1}$, $f_{\ell}=f_{n}*\varphi_{\ell}=(f_{\ell_{n}}+g_{n})*\varphi_{\ell}$ where $f_{\ell_{n}}=f_{n-1}*\varphi_{\ell_{n}}$. Then we  write
\begin{align*}&\Lambda f_\ell\nabla^{\bot}(z_{n+1}-z_n) +\Lambda (z_{n+1}-z_n)\nabla^{\bot}f_\ell\\
&=\Lambda f_{\ell_{n}}*\varphi_\ell\nabla^{\bot}(z_{n+1}-z_n) +\Lambda (z_{n+1}-z_n)\nabla^{\bot}f_{\ell_{n}}*\varphi_\ell
\\&\quad+\Lambda {g}  _n*\varphi_\ell\nabla^{\bot}(z_{n+1}-z_n) +\Lambda (z_{n+1}-z_n)\nabla^{\bot}{g}  _n*\varphi_\ell=:I_1+I_2,
\end{align*}
where $I_1$ denotes  the second line and $I_2$ denotes  the third line.
Additionally, we introduce  further decomposition
of the difference $z_{n+1}-z_n$. Namely,
\begin{align}\label{decom}z_{n+1}-z_n=\sum_{l=0}^N\sum_k 1_{C_l}(k)\widehat{(z_{n+1}-z_n)}(k)e_k,
\end{align}
where $e_k(x)=\frac1{(2\pi)^2}e^{ik\cdot x}$, $N=\lfloor(b-1)/\zeta\rfloor+1$, $C_0=\{|k|_\infty\leq 10\lambda_n\}$ and for $l=1,\dots N$ we define
$$C_l=\{10\mu^{l-1}_n\lambda_n<|k|_\infty\leq 10\mu^{l}_n\lambda_n\}.$$
These sets split the Fourier support of $z_{n+1}-z_{n}$ into annuli which permits us to use either the  upper or the lower bound on the Fourier support in various estimates so that we are able to balance the required conditions on parameters.
Here and in the sequel we denote $|k|_\infty=\max\{k^1,k^2\}$. By definition we have
\begin{align}\label{estfl}
\|\nabla f_{\ell_n}\|_{L^\infty}+\|\Lambda f_{\ell_n}\|_{L^\infty}\lesssim \lambda_n^{1/2}\|f_{\ell_n}\|_{B^{1/2}_{\infty,1}} ,
\end{align}
and
\begin{align}\label{estz}
\left\|\sum_k 1_{C_l}(k)\widehat{(z_{n+1}-z_n)}(k)e_k\right\|_{C^{1+}}\lesssim (\lambda_n\mu_n^l)^{1+2\kappa+\alpha-\gamma_1/2}\|z_{n+1}-z_n\|_{C^{\gamma_1/2 {-\alpha}-\kappa}}.
\end{align}
Here we used \eqref{fouT} and \cite[Lemma 8.7]{GP17} to deduce $\|\mathcal{F}^{-1}_{\mT^2}1_{C_l}\|_{L^1(\mT^2)}\lesssim \log ^2(\mu_n^l\lambda_n).$

By the decomposition of $z_{n+1}-z_n$ in \eqref{decom} we obtain the decomposition of $I_1=\sum_{l=0}^N I_1^l$. That is, each $I_1^l$ is given by $I_1$ with $z_{n+1}-z_n$ replaced by $\sum_k 1_{C_l}(k)\widehat{(z_{n+1}-z_n)}(k)e_k$.
As  $I^0_1$ is localized outside a ball of radius  $\lambda_n/3$,  we use  \eqref{estfl}, \eqref{estz}, H\"{o}lder's inequality and \eqref{inductionv ps} on the level $n$ and Proposition \ref{fe z} to obtain that the part of $\$q_z\$_{X,r}$ corresponding to  $I^0_1$ is controlled by
\begin{align*}
\$ I_{1}^{0}\$_{C^{-1+},r} &\lesssim \lambda_{n}^{-1+\kappa}\$ I_{1}^{0}\$_{L^{\infty},r}
\\&\lesssim \lambda_n^{-1+\kappa}\lambda_n^{1/2}\$f_{\ell_{n}}\$_{B^{1/2}_{\infty,1},2r}\lambda_{n}^{1+2\kappa {+\alpha}-\gamma_1/2}\$z_{n+1}-z_n\$_{C^{\gamma_1/2 {-\alpha}-\kappa},2r}
	\\&\lesssim \lambda_n^{-1/2+\kappa}\lambda_{n}^{1+2\kappa-\gamma_1/2 {+\alpha}} M_L\leq \frac1{35(N+1)}r_{n+1},
\end{align*}
where we used $1/2+\beta b+3\kappa<\frac{\gamma_1}2-\alpha$.

 Moreover, as each $I^l_1$, $l=1,\dots, N$, is localized outside a ball of radius $\mu_n^{l-1}\lambda_n$,
\eqref{estfl}, \eqref{estz} and H\"{o}lder's inequality  lead to
\begin{align*}
\$I_{1}^{l}\$_{C^{-1+},r}
&\lesssim (\lambda_n\mu_n^{l-1})^{-1+\kappa}\lambda_n^{1/2}\$f_{\ell_{n}}\$_{B^{1/2}_{\infty,1},2r}(\lambda_{n}\mu_n^l)^{1+2\kappa {+\alpha}-\gamma_1/2}\$z_{n+1}-z_n\$_{C^{\gamma_1/2 {-\alpha}-\kappa},2r}
	\\&\lesssim \lambda_n^{-1/2+\kappa}\lambda_{n}^{1+2\kappa-\gamma_1/2{+\alpha}+\zeta} M_L\leq \frac1{35(N+1)}r_{n+1}
\end{align*}
provided \eqref{inductionv ps} for $f_{ n-1}$ and Proposition \ref{fe z}. We also used the condition on the parameters to deduce that  the exponent  of $\mu_n$ is smaller than $1$ and that $1/2+\beta b+3\kappa<\frac{\gamma_1}2-\alpha-\zeta$.

In order to estimate $I_2$, we decompose as follows
\begin{align*}I_2&=5\lambda_{n}[g_n*\varphi_\ell\nabla^{\bot}(z_{n+1}-z_n) + (z_{n+1}-z_n)\nabla^{\bot}g_n*\varphi_\ell]\\
&\quad+(\Lambda g_n-5\lambda_{n}g_n)*\varphi_\ell\nabla^{\bot}(z_{n+1}-z_n) \\
&\quad+(\Lambda-5\lambda_{n})(z_{n+1}-z_n)\nabla^{\bot}g_n*\varphi_\ell\\
&=:I_{21}+I_{22}+I_{23} \approx I_{22}+I_{23},
\end{align*}
where the last step follows from $g_n*\varphi_\ell\nabla^{\bot}(z_{n+1}-z_n) + (z_{n+1}-z_n)\nabla^{\bot}g_n*\varphi_\ell=\nabla^{\bot}((z_{n+1}-z_n)g_n*\varphi_\ell)$.
We employ again the decomposition \eqref{decom} and get $I_{22}=\sum_{l=0}^NI_{22}^l$, where each $I_{22}^l$ is given by $I_{22}$ with $z_{n+1}-z_n$ replaced by $\sum_k 1_{C_l}(k)\widehat{(z_{n+1}-z_n)}(k)e_k$.
By Lemma \ref{lem:21} we obtain
\begin{align*}\Lambda   \tilde{g}  _n-5\lambda_{n}  \tilde{g}  _n=\sum_{j=1}^2l_j\cdot \nabla a_{j,n}\sin(5\lambda_nl_j\cdot x)+T^{(1)}_{5\lambda_nl_j}a_{j,n}\cos(5\lambda_nl_j\cdot x)+T^{(2)}_{5\lambda_nl_j}a_{j,n}\sin(5\lambda_nl_j\cdot x),
\end{align*}
which by Lemma \ref{lem:33} implies that
 \begin{align*}\|\Lambda g_n-5\lambda_{n}g_n\|_{L^\infty}\lesssim ({\lambda_{n-1}} +\lambda_n^{-1}\mu_n^2+\lambda_n^{-2}\mu_n^3)\lambda_n^{-1/2} \tilde{\chi}^{1/2},
 \end{align*}
 where we computed the derivative of $a_{j,n}$ and used the fact that $q_{\ell_n}$ is localized on frequencies of order at most  $\lambda_{n-1}$.
 Thus by \eqref{estz} and H\"{o}lder's inequality we obtain
\begin{align*}
&\$I^{0}_{22}\$_{C^{-1+},r}
\\&\lesssim({\lambda_{n-1}} +\lambda_n^{-1}\mu_n^2+\lambda_n^{-2}\mu_n^3)\lambda_n^{-1/2}(\$q_{n-1}\$_{X,r}^{1/2}+r_{n-1}^{1/2})\lambda_{n}^{1+2\kappa-\gamma_1/2 {+\alpha}}\$z_{n+1}-z_n\$_{C^{\gamma_1/2 {-\alpha}-\kappa},2r}
\\&\lesssim \lambda_n^{1/b-1/2}\lambda_{n}^{1+2\kappa-\gamma_1/2 {+\alpha}}M_L\leq \frac1{35(N+1)}r_{n+1},
\end{align*}
 where we used the condition $1/2+2\kappa+1/b+\beta b<\frac{\gamma_1}2-\alpha.$

 Moreover, as $I^l_{22}$ is localized in ball of radius $\lambda_n\mu_n^{l-1}$ we obtain by H\"{o}lder's inequality and \eqref{estz}
\begin{align*}
&\$I^{l}_{22}\$_{C^{-1+},r}
\\&\lesssim(\lambda_n\mu_n^{l-1})^{-1+\kappa}({\lambda_{n-1}} +\lambda_n^{-1}\mu_n^2+\lambda_n^{-2}\mu_n^3)\lambda_n^{-1/2}(\$q_{n-1}\$_{X,r}^{1/2}+r_{n}^{1/2})\\
&\qquad\qquad\times(\lambda_{n}\mu_n^{l})^{1+2\kappa-\gamma_1/2 {+\alpha}}\$z_{n+1}-z_n\$_{C^{\gamma_1/2 {-\alpha}-\kappa},2r}
	\\&\lesssim \lambda_n^{1/b-1/2+\zeta}\lambda_{n}^{3\kappa-\gamma_1/2 {+\alpha}}M_L\leq \frac1{35(N+1)}r_{n+1}.
\end{align*}
Here we used the conditions on parameters to deduce $\lambda_n^{-1}\mu_n^2+\lambda_n^{-2}\mu_n^3\leq1$ and that the other exponent of  $\mu_n$ is smaller than 1 and that $3\kappa+\zeta+1/b+\beta b<\frac{\gamma_1}2+\frac12-\alpha.$

Finally,  we concentrate on $I_{23}$ which requires a different decomposition. In particular,
we let $N_0=\lfloor b/\zeta\rfloor+1$ and for $l=1,\dots,N_0$ we define
$$A_l=\{k:|k+5\lambda_nl_1|_\infty\leq 5\mu^l_n\},\quad \bar A_l=\{k:|k-5\lambda_nl_1|_\infty\leq 5\mu^l_n\},$$
$$B_l=\{k:|k+5\lambda_nl_2|_\infty\leq 5\mu^l_n\},\quad \bar B_l=\{k:|k-5\lambda_nl_2|_\infty\leq 5\mu^l_n \}.$$
Let $D_l=A_l\cup \bar A_l\cup B_l\cup \bar B_l$.
Then it holds
\begin{align*}I_{23}&=\sum_{k}1_{D_1}(k)(|k|-5\lambda_n)\widehat{(z_{n+1}-z_n)}(k)e_k\nabla^{\bot}g_n*\varphi_\ell
\\&\quad+\sum_{l=1}^{N_0-1}\sum_{k}1_{D_l^c\cap D_{l+1}}(k)(|k|-5\lambda_n)\widehat{(z_{n+1}-z_n)}(k)e_k\nabla^{\bot}g_n*\varphi_\ell
\\&=:I_{23}^0+\sum_{l=1}^{N_0-1}I_{23}^l.
\end{align*}
Then we obtain on $D_1$ $|k|\leq 5\lambda_n+C\mu_n$ which by Proposition \ref{fe z} implies that
\begin{align*}
\$I_{23}^0\$_{C^{-1+},r}
&\lesssim \$\sum_{k}1_{D_1}(k)(|k|-5\lambda_n)\widehat{(z_{n+1}-z_n)}(k)e_k\$_{H^{\kappa/2},2r}\$\nabla^{\bot}  \tilde{g}  _n*\varphi_\ell\$_{C^\kappa,2r}
\\&\lesssim \lambda_n^{\zeta}\$z_{n+1}-z_n\$_{H^{\kappa/2},2r}\lambda_n^{1/2+\kappa}\$  \tilde{g}  _n\$_{C^{1/2},2r}
\\&\lesssim\lambda_n^{\zeta+1/2+\alpha-\gamma_1/2+3\kappa}M_L\leq \frac1{35N_0}r_{n+1}.
\end{align*}
Here we used \eqref{gn} for $  \tilde{g}  _n$ and  $\frac12+3\kappa+\beta b<\frac{\gamma_1}2-\alpha-\zeta.$

On $D_l^c$ we have
$$
|k+5\lambda_nl_1|_\infty>5\mu^l_n,\quad |k-5\lambda_nl_1|_\infty>5\mu^l_n,\quad |k+5\lambda_nl_2|_\infty>5\mu^l_n,\quad |k-5\lambda_nl_2|_\infty>5\mu^l_n
$$
 which by Remark \ref{r:1} implies that  the term $I_{23}^l$ is localized outside a ball of radius $\mu_n^l$. Indeed,  by Remark \ref{r:1} we know that $\mathcal{F}_{\mT^2}  \tilde{g}  _n(m)\neq0$ only if one of the following four cases happens
 $$
 |m+5\lambda_nl_1|\leq \mu_n,\quad |m-5\lambda_nl_1|\leq\mu_n
 ,\quad |m+5\lambda_nl_2|\leq\mu_n,\quad |m-5\lambda_nl_2|\leq\mu_n.
 $$ In any case we obtain $|k+m|>4\mu^l_n$.  Thus,
\begin{align*}
&\|I^l_{23}\|_{C^{-1+}}\lesssim \mu_n^{-l+\kappa}\left\|\sum_{k}1_{A_{l+1}\cup \bar A_{l+1}\cup  B_{l+1}\cup \bar B_{l+1}}(k)(|k|-5\lambda_n)\widehat{(z_{n+1}-z_n)}(k)e_k\right\|_{L^\infty}\|\nabla^{\bot}  \tilde{g}  _n*\varphi_\ell\|_{L^\infty}.
\end{align*}
Here we used \eqref{fouR} and \cite[Lemma 8.7]{GP17} to deduce $\|\mathcal{F}^{-1}_{\mathbb{T}^2}1_{D_l^c}\|_{L^1(\mT^2)}\lesssim (\log \mu_n)^8$.
To estimate the right hand side, we distinguish two cases: $l+1<1/\zeta\leq l+2$ and $l\geq 1/\zeta-1$.

  Case I. If $l+1<1/\zeta\leq l+2$ then  $10\mu_n^{l+1}<\lambda_n$ and we can take $a$ large enough such that $10\mu_n^{l+1}<\frac14\lambda_n$. As the sets $A_{l+1},\bar A_{l+1}, B_{l+1},\bar B_{l+1}$ are disjoint,
 we can consider them separately. In what follows, we only consider the case of $1_{A_{l+1}}$, the others being similar. Let  $\phi\in C^\infty(\mR^2)$ such that $\phi(x)=1$ for $|x|\leq 1$ and $\phi(x)=0$ for $|x|>2$. Then
\begin{align*}&\sum_{k}1_{A_{l+1}}(k)(|k|-5\lambda_n)\widehat{(z_{n+1}-z_n)}(k)e_k\\
&=\sum_{k}1_{A_{l+1}}(k)\phi\bigg(\frac{|k+5\lambda_n l_1|}{10\mu_n^{l+1}}\bigg)(|k|-5\lambda_n)\widehat{(z_{n+1}-z_n)}(k)e_k
\end{align*}
Similarly to the proof of Lemma~\ref{lem:33} (see \cite[Lemma 3.3]{CKL21}) it holds  for $m(k)=\phi(\frac{k+5\lambda_n l_1}{10\mu_n^{l+1}})(|k|-5\lambda_n)$ that
\begin{align}\label{fm}\|\mathcal{F}_{\mR^2}^{-1}m\|_{L^1(\mR^2)}\lesssim \mu_n^{l+1}.
\end{align}
Indeed,
\begin{align*}\mathcal{F}^{-1}_{\mR^2}m(x)&=(2\pi)^{-2}\int \phi\bigg(\frac{k+5\lambda_n l_1}{10\mu_n^{l+1}}\bigg)(|k|-5\lambda_n)e^{ik\cdot x}\dif k
\\&=\lambda_n(2\pi)^{-2}\int \phi\bigg(\frac{k}{10\mu_n^{l+1}}\bigg)\left(\bigg|\frac{k}{\lambda_n}-5 l_1\bigg|-5\right)e^{ik\cdot x}\dif k e^{-5i\lambda_nl_1\cdot x}
\\&=\lambda_n(2\pi)^{-2}\int \phi\bigg(\frac{k}{10\mu_n^{l+1}}\bigg)\phi_1\bigg(\frac{k}{\lambda_n}\bigg)e^{ik\cdot x}\dif k e^{-5i\lambda_nl_1\cdot x}
\end{align*}
Here $\phi_1(k)=|k-5l_1|-5$ is smooth and $|\phi_1(k)|\lesssim |k|$  for $|k|\leq 1/2$. Thus \eqref{fm} follows from the following
\begin{align*}\bigg\|\mathcal{F}^{-1}_{\mR^2}\bigg(\phi\bigg(\frac{\cdot}{10\mu_n^{l+1}}\bigg)\phi_1\bigg(\frac{\cdot}{\lambda_n}\bigg)\bigg)\bigg\|_{L^1}=\bigg\|\mathcal{F}_{\mR^2}^{-1}\bigg(\phi_1\bigg(\frac{10\mu_n^{l+1}\cdot}{\lambda_n}\bigg)\phi\bigg)\bigg\|_{L^1}
\lesssim \lambda_n^{-1}\mu_n^{l+1}.
\end{align*}
Hence we obtain by \eqref{fouR} and \eqref{fouT} that
\begin{align*}
&\left\|\sum_{k}1_{A_{l+1}}(k)(|k|-5\lambda_n)\widehat{(z_{n+1}-z_n)}(k)e_k\right\|_{L^\infty}
\\&\lesssim \mu_n^{l+1}\left\|\sum_{k}1_{A_{l+1}}(k)\widehat{(z_{n+1}-z_n)}(k)e_k\right\|_{L^\infty}
\\&\lesssim\mu_n^{l+1+\kappa}\|z_{n+1}-z_n\|_{L^\infty},
\end{align*}
where we used  \cite[Lemma 8.7]{GP17} to deduce $\|\mathcal{F}^{-1}_{\mathbb{T}^2}1_{A_{l+1}}\|_{L^1(\mT^2)}\lesssim (\log \mu_n)^2$.
Thus by H\"{o}lder's inequality it holds that
\begin{align*}
\$I_{23}^l\$_{C^{-1+},r}
&\lesssim  \mu_n^{-l+2\kappa}\mu_n^{{l+1}}\$z_{n+1}-z_n\$_{L^\infty,2r}\lambda_n^{1/2}\$  \tilde{g}  _n\$_{B^{1/2}_{\infty,1},2r}
\\&\leq \mu_n^{-l+2\kappa}\mu_n^{{l+1}}\$z\$_{C^{\gamma_1/2 {-\alpha}-\kappa},2r}\lambda_{n}^{-\gamma_1/2 {+\alpha}+2\kappa}\lambda_n^{1/2}\$  \tilde{g}  _n\$_{B^{1/2}_{\infty,1},2r}\leq\frac1{35N_0}r_{n+1}.
\end{align*}
Here we used Proposition \ref{fe z} and \eqref{gn} for $  \tilde{g}  _n$ and $\frac12+3\kappa+\beta b<\frac{\gamma_1}2 {-\alpha}-\zeta.$

Case II. For the case that $l\geq1/\zeta-1$  we obtain $\lambda_n\leq \mu_n^{l+1}$ and
\begin{align*}
\$I_{23}^l\$_{{C^{-1+}},r}&\lesssim  \mu_n^{-l+\kappa}\$\sum_{k}1_{D_{l+1}}(|k|-5\lambda_n)\widehat{(z_{n+1}-z_n)}(k)e_k\$_{L^\infty,2r}\|\nabla^{\bot}  \tilde{g}  _n*\varphi_\ell\|_{L^\infty,2r}.
\end{align*}
Then we have
\begin{align*}
&\$\sum_{k}1_{D_{l+1}}(|k|-5\lambda_n)\widehat{(z_{n+1}-z_n)}(k)e_k\$_{L^\infty,2r}
\\&\lesssim  \$\sum_{k}1_{D_{l+1}}\widehat{(z_{n+1}-z_n)}(k)e_k\$_{C^{1+},2r}+\lambda_n\$\sum_{k}1_{D_{l+1}}\widehat{(z_{n+1}-z_n)}(k)e_k\$_{L^\infty,2r}
\\&\lesssim(\mu_{n}^{(l+1)(1-\gamma_1/2+\alpha+2\kappa)}+\lambda_n\lambda_n^{-\gamma_1/2+\alpha+2\kappa})\$z\$_{C^{\gamma_1/2-\alpha-\kappa},2r},
\end{align*}
where we used \eqref{fouT} and \cite[Lemma 8.7]{GP17} to deduce $\|\mathcal{F}_{\mT^2}^{-1}1_{D_{l+1}}\|_{L^1(\mT^2)}\lesssim (\log \mu_n)^8$ and in the last step for the first term we used the upper bound of Fourier support and for the second bound we used the lower bound of Fourier support.
This implies that
\begin{align*}
\$I_{23}^l\$_{{C^{-1+}},r}&\lesssim  \lambda_n^{\zeta-\frac{\gamma_1}2+\alpha+3\kappa+\frac12}M_L
\leq \frac1{35N_0}r_{n+1}.
\end{align*}
Here, we used $\mu_n^{-l-1}\leq \lambda_n^{-1}$ and the condition $\frac12+3\kappa+\beta b<\frac{\gamma_1}2 -\alpha-\zeta$.

To summarize the results of this subsection, we proved that
$$
\$q_{z}\$_{X,r}\leq \frac17 r_{n+1}
$$
and this will be combined with the estimates in the next subsection to deduce \eqref{iteration R} and \eqref{iteration R1}.

\subsubsection{Estimate of the other parts of $q_{n+1}$}

In the sequel, we are mostly inspired by \cite{HZZ22} and \cite{CKL21}, but some differences appear here in the treatment of the additional terms.

Let us start with  $q_{M}$. If $t\geq 2^{-n}$, we proceed similarly to  the proof of Proposition~3.1 in \cite{CKL21} with $r_n$ replaced by $\tilde\chi$  and obtain
$$\|\chi^{2}\Delta^{-1}\nabla\cdot (\Lambda \tilde g_{n+1}\nabla^{\perp}\tilde g_{n+1}+\nabla q_{\ell})\|_X\lesssim \log \mu_{n+1}(\mu_{n+1}^{-1}\lambda_n)^2\tilde{\chi}+\log \mu_{n+1}(\mu_{n+1}\lambda_{n+1}^{-1})^2\tilde{\chi}+\lambda_{n}\lambda_{n+1}^{-1}\tilde{\chi},
$$
which by \eqref{iteration R} and \eqref{iteration R1} on the level $n$ in particular implies that
 $$\left(\sup_{t\geq 2^{-n}}\mathbf{E}\sup_{ s\in[t, t+1]}\|\chi^{2}\Delta^{-1}\nabla\cdot (\Lambda \tilde g_{n+1}\nabla^{\perp}\tilde g_{n+1}+\nabla q_{\ell})\|_{X}^r\right)^{1/r}\leq\frac17r_{n+1}.
$$
In particular,  we used the fact  that $\zeta>1/b+\beta/2$, the implicit constants  depend on $M_0$ and we choose $a$ large enough to absorb them.

If $t\in[0,2^{-n}]$ then we have the additional term $(1-\chi^{2})q_{\ell}$, which can be controlled as
$$
\|(1-\chi^2)q_\ell\|_{X}\leq\sup_{s\in[t-\ell,t]}\|q_n(s)\|_X,
$$
hence by \eqref{iteration R1} on the level $n$ we obtain
$$
\left(\E\sup_{s\in[0,2^{-n}]}\|(1-\chi^2)q_\ell\|_{X}^{r} \right)^{1/r}\leq M_{L}+\sum_{k=1}^{n}r_{k}.
$$

The second term which requires the splitting into $t\geq 2^{-n}$ and $t\in [2^{-n-1},2^{-n}]$ is $q_{I}$, the remaining terms can be controlled directly for $t\geq0$. In particular, one part of $q_{I}$ is bounded directly for $t\geq 0$
 using  \eqref{esta1} as
$$
\begin{aligned}
	\$\chi\Delta^{-1}\nabla \cdot \partial_t\mathcal{R}   \tilde{g}_{n+1}  \$_{X,r}&\lesssim  \frac{\log\lambda_{n+1}} {\lambda_{n+1}} \lambda_{n+1}^{-1/2}r_{n}^{-1/2}\lambda_{n+1}\$q_n\$_{X,r}\leq\frac1{14}r_{n+1}.
\end{aligned}
$$
and the part containing the time derivative of $\chi$ appears for $t\in[2^{-n-1},2^{-n}]$ and is bounded by
$$\$\Delta^{-1}\nabla \cdot \mathcal{R}   \tilde{g}_{n+1}  \partial_t\chi\$_{{X},r}\lesssim 2^{n+1}\lambda_{n+1}^{-3/2}M_L^{1/2}\log\lambda_{n+1}\leq\frac1{14}r_{n+1}.$$

For $q_T$, we proceed
similarly to  \cite{CKL21} and using the definition of $  \tilde{g}  _{n+1}$  we deduce
\begin{align*}
	&\|\Delta^{-1}\nabla\cdot(\chi\Lambda   \tilde{g}  _{n+1}\nabla^\bot f_{{\ell}}+\chi\Lambda f_{{\ell}}\nabla^\bot   \tilde{g}  _{ n+1})\|_X
	\\&\leq\log\lambda_{n+1} \|  \tilde{g}_{n+1}  \|_{L^\infty}(\|\Lambda f_{{\ell}}\|_{L^\infty}+\|\nabla^\bot f_{{\ell}}\|_{L^\infty})\lesssim \log\lambda_{n+1}\|  f_n  \|_{C_tB^{1/2}_{\infty,1}}\lambda_n^{1/2}(\tilde{\chi}\lambda_{n+1}^{-1})^{1/2},
\end{align*}
which by H\"{o}lder's inequality implies that
\begin{align*}
	&\$\Delta^{-1}\nabla\cdot(\chi\Lambda   \tilde{g}  _{n+1}\nabla^\bot f_{{\ell}}+\chi\Lambda f_{{\ell}}\nabla^\bot   \tilde{g}  _{ n+1})\$_{X,r}
	\\&\lesssim \log\lambda_{n+1}\$  f_n  \$_{B^{1/2}_{\infty,1},2r}\lambda_n^{1/2}\lambda_{n+1}^{-1/2} {(\$q_n\$^{1/2}_{X,r}+r_n^{1/2})}\leq\frac1{7}r_{n+1},
\end{align*}
provided \eqref{inductionv ps} for $f_n$ and $1/2>1/({2b})+\beta.$

For $q_D$ we  recall that the support of the Fourier transform of $  \tilde{g}_{n+1}  $ is contained in an annulus. Hence application of \eqref{gn} leads to
$$
\$q_D\$_{X,r}\lesssim \lambda_{n+1}^{\gamma-1}\sqrt{M_L\lambda_{n+1}^{-1}}\leq\frac17r_{n+1},$$
provided $\beta<3/2-\gamma$. 

For $q_N$ we have
\begin{align*}\nabla q_N&=(P_{\leq\lambda_{n+1}}-P_{\leq\lambda_{n}})(:\Lambda z\nabla^\bot z:+\Lambda^{\gamma_1} \mathcal{R} z-\Lambda^{\gamma} \mathcal{R} z)
\\&\quad +P_{\leq\lambda_{n}}(:\Lambda z\nabla^\bot z:+\Lambda^{\gamma_1}\mathcal{R} z-\Lambda^{\gamma} \mathcal{R}z)-P_{\leq\lambda_{n}}(:\Lambda z\nabla^\bot z:+\Lambda^{\gamma_1}  \mathcal{R}z-\Lambda^{\gamma}\mathcal{R}  z)*\varphi_\ell
\\&=:\nabla q_N^1+\nabla q_N^2.
\end{align*}
Then we obtain
\begin{align*}
	\| q^1_N\|_{X} &\lesssim \lambda^{1-\gamma_1 {+2\alpha}+2\kappa}_n\|:\Lambda z\nabla^\bot z:\|_{C^{\gamma_1 {-2\alpha}-2-\kappa}}+\lambda^{\gamma_1/2 {+\alpha}-1+2\kappa}_n\|\Lambda^{\gamma_1}  z\|_{C^{-\gamma_1/2 {-\alpha}-\kappa}}
\\&\quad + \lambda^{\gamma-\gamma_1/2 {+\alpha}-1+2\kappa}_n\|\Lambda^{\gamma}  z\|_{C^{-\gamma+\gamma_1/2 {-\alpha}-\kappa}},
\end{align*}
and
\begin{align*}
	\| q^2_N\|_{X} &\lesssim \ell^{1-1/\gamma_1 {-2\alpha/\gamma_1}-\kappa} \|:\Lambda z\nabla^\bot z:\|_{C_t^{1-1/\gamma_1 {-2\alpha/\gamma_1}-\kappa}C^{-1+\kappa}}
\\&\quad+\ell^{1/\gamma_1 {-\alpha/\gamma_1}-1/2-\kappa}\|\Lambda^{\gamma_1}  z\|_{C^{1/\gamma_1 {-\alpha/\gamma_1}-1/2-\kappa}_tC^{-1+\kappa}}
\\&\quad+\ell^{\frac12-\gamma/\gamma_1 {-\alpha/\gamma_1}+1/\gamma_1-\kappa}\|\Lambda^{\gamma}  z\|_{C^{\frac12-\gamma/\gamma_1 {-\alpha/\gamma_1}+1/\gamma_1-\kappa}_tC^{-1+\kappa}}, \end{align*}
which by Proposition \ref{fe z} and Proposition \ref{lem:re} implies that
$$\$q_N\$_{X,r}\leq \frac17r_{n+1},$$
provided  $ (1-\frac{\gamma_1}{2} {-\alpha})\wedge (\gamma_1-1-2\alpha)\wedge (1+\frac{\gamma_1}{2}-\gamma-\alpha)>b \beta+2\kappa$ and
$(1-\frac1{\gamma_1}-\frac{2\alpha}{\gamma_1})\wedge (\frac{1}{\gamma_1} -\frac{\alpha}{\gamma_1}-\frac12)\wedge (\frac12-\frac{\gamma+\alpha-1}{\gamma_1})>\beta+\kappa.$

Moreover,  we get
\begin{align*}
\|R_{\rm{com}}\|_{X}&\lesssim \log \lambda_{n}\ell^{1/2}_{n+1}\|\Lambda   f_n  \nabla^\bot   f_n  \|_{C_t^{1/2}L^\infty}+\log \lambda_{n}\ell^{1/2}_{n+1}(\|  f_n  \|_{C_t{B^1_{\infty,1}}}+\|z_n\|_{C_t{B^1_{\infty,1}}})\|  f_n  \|_{C_t^{1/2}{B^1_{\infty,1}}}
\\&\quad+\log \lambda_{n} \ell_{n+1}^{1/2-\kappa} \|z_n\|_{C^{1/2-\kappa}_t{B^1_{\infty,1}}}\|  f_n  \|_{C_t{B^1_{\infty,1}}}
\end{align*}
which by Proposition \ref{fe z} and \eqref{inductionv ps}, \eqref{eq:C1} for $f_n$ and H\"{o}lder's inequality implies
\begin{align*}
\$R_{\rm{com}}\$_{X,r}\lesssim {\log \lambda_{n} \ell^{1/2}_{n+1}\lambda^{3/2}_nM_L+\log \lambda_{n} \ell_{n+1}^{1/2-\kappa}{\lambda_n^{3/2}M_L}\leq\frac17 r_{n+1},}
\end{align*}
where we used that  $\beta+\frac{3}{2b}+\kappa<\frac12$.

Altogether,  the above estimates give the control of $q_{n+1}$ for $t\in [2^{-n},\infty)$ and  \eqref{iteration R} follows.
In addition,
  for small times we conclude that
$$\left(\mathbf{E}\sup_{s\in[0,2^{-n}]}\|q_{n+1}(s)\|^r_X\right)^{1/r}\leq r_{n+1}+M_{L}+\sum_{k=1}^{n}r_{k}.$$
Hence \eqref{iteration R1} holds and the proof of Proposition~\ref{p:iteration} is therefore complete.

\section{Initial value problem}
\label{s:1.1}

 \bt\label{thm:6.1}
Let $\delta>0$ and $\kappa>0$ be arbitrary. For any  $\mathcal{F}_0$-measurable initial condition $\theta_0$ with zero mean satisfying \eqref{eq:u0} for some $r\geq1$ and $\eta>1/2$,
there exist infinitely many,  $(\mathcal{F}_t)$-adapted, non-Gaussian, analytically weak solutions $\theta$ to \eqref{eq1} in the sense of Definition~\ref{d:1}, with the initial condition $\theta_0$ and  belonging to
 $$L^p_{\rm{loc}}([0,\infty);B_{\infty,1}^{-1/2})\cap C([0,\infty),B_{\infty,1}^{-1/2-\delta})\cap C^{1/2-\kappa}([0,\infty),B^{-1-\delta}_{\infty,1})\quad \mathbf{P}\text{-a.s. for all } p\in[1,\infty) .$$
Furthermore, there exists $\vartheta>1/2$ so that
\begin{equation}\label{est:fn}
\sup_{s\geq4}\left[\mathbf{E}\sup_{t\in[s,s+1]}\|\theta\|_{B^{\vartheta-1}_{\infty,1}}^{2r}\right]^{1/2r}+\$\theta\$_{C_t^{1/2-\kappa}B^{-1-\delta}_{\infty,1},2r}\lesssim 1
\end{equation}
and  for every $\varepsilon>0$ the solutions can be constructed so that
\begin{equation}\label{eq:32}
\$\theta-\Lambda z\$_{B^{-1/2-\delta}_{\infty,1},2r}\leq \varepsilon.
\end{equation}

\et

\begin{proof}
We repeatedly apply Proposition~\ref{p:iteration} and obtain a sequence of  $(\mathcal{F}_t)$-adapted processes $(  f_n  ,q_{n})$, $n\in\N_{0}$, such that  $  f_n   \to f$ in
$L^{2r}(\Omega,C([0,\infty);B^{1/2-\delta}_{\infty,1}))$  as a consequence of  \eqref{induction w}. Moreover, using {\eqref{iteration ps}, \eqref{iteration ps1}, \eqref{iteration ps2}} we
have for every $p\in[1,\infty),T>1$
\begin{align*}
  & \left[\mathbf{E}\left(\int_0^{T}  \| f_{  n + 1} -   f_n   \|_{B^{1/2}_{\infty,1}}^p \dif t\right)^{2r/p}\right]^{1/2r}
\\ &\lesssim\left[\mathbf{E}\left(\int_{2^{-n+2}}^{T} \| f_{  n + 1} -   f_n  \|_{B^{1/2}_{\infty,1}}^p \dif t \right)^{2r/p}\right]^{1/2r}+
   \left[\mathbf{E}\left(\int_{2^{-n-1}}^{2^{-n+2}} \| f_{  n + 1} -   f_n   \|_{B^{1/2}_{\infty,1}}^p \dif t\right)^{2r/p}\right]^{1/2r}
   \\&\quad+\left[\mathbf{E}\left(\int_0^{2^{-n-1}} \| f_{  n + 1} -   f_n   \|_{B^{1/2}_{\infty,1}}^p \dif t\right)^{2r/p}\right]^{1/2r}
\\&\lesssim 2^{- n/p} M_L^{1/2}  + r_{n+1}^{1/2}+r_n^{1/2}.
\end{align*}
This implies that  the sequence $  f_n  $, $n\in\N_{0}$, converges in $L^{2r}(\Omega,L^{p}(0,T;B_{\infty,1}^{1/2}))$ for all $p\in[1,\infty)$.
Accordingly, $  f_n  \to f$ also in $L^{2r}(\Omega,L^p_{\rm{loc}}([0,\infty);B_{\infty,1}^{1/2}))$.
Furthermore, by \eqref{iteration R}, \eqref{iteration R1} it follows  for all $p\in[1,\infty)$, $T>1$ that
	$$
	\left[\mathbf{E}\left(\int_0^{T}\|q_n(t)\|^{p}_{X}\dif t\right)^{r/p}\right]^{1/r}\lesssim r_{n}+M_L2^{-n/p} \to 0, \quad\mbox{as}\quad n\to\infty.
	$$
	
Thus, by \eqref{besov} and similarly as in  \cite{CKL21} the process $\theta=\Lambda f +\Lambda z$ satisfies \eqref{eq1}  in the analytically weak sense.
More precisely, we define $\theta_n=\Lambda   f_n  +\Lambda z_n$ and from \eqref{induction ps} we obtain for any $\psi\in C^\infty$
\begin{align*}
\langle \theta_n(t)-P_{\leq {\lambda_n}}\theta_0,\psi\rangle=& \int_0^t\langle (\theta_n-\Lambda z_n)
   \mathcal{R}^\bot(\theta_n-\Lambda z_n)+(\theta_n-\Lambda z_n)
   \mathcal{R}^\bot \Lambda z_n+\Lambda z_n
   \mathcal{R}^\bot(\theta_n-\Lambda z_n)\\&+P_{\leq \lambda_n} :(\Lambda z\nabla^\bot z):
   +\nabla q_n, \nabla\psi\rangle
  -\langle
    \Lambda^{\gamma } \theta_{n}
,\psi\rangle\dif s+ \int_0^t\langle
   P_{\leq \lambda_n}\Lambda^\alpha \psi,\dif B_s\rangle.
\end{align*}
Since $\< \mathcal{R}_j f,g\>=-\<f,\mathcal{R}_jg\>$ the above  rewrites  as
\begin{equation}\label{thetan}
\begin{aligned}
	\langle \theta_n(t)-P_{\leq {\lambda_n}}\theta_0,\psi\rangle=& \int_0^t-\frac12\langle \theta_n-\Lambda z_n,[\mathcal{R}^\perp,\nabla \psi] (\theta_n-\Lambda z_n)\rangle-\frac12\langle \Lambda z_n,[\mathcal{R}^\perp,\nabla \psi] (\theta_n-\Lambda z_n)\rangle
	\\&-\frac12\langle (\theta_n-\Lambda z_n),[\mathcal{R}^\perp,\nabla \psi]\Lambda z_n \rangle-\frac12\langle \Lambda z,[\mathcal{R}^\perp,\nabla P_{\leq \lambda_n}\psi]\Lambda z \rangle
	\\&-\langle \nabla q_n, \nabla\psi\rangle
-\langle
    \Lambda^{\gamma } \theta_{n}
,\psi\rangle\dif s+ \int_0^t\langle
   P_{\leq \lambda_n}\Lambda^\alpha \psi,\dif B_s\rangle.
   \end{aligned}
\end{equation}
Here, we used the fact that $$-\frac12\langle \Lambda z,[\mathcal{R}^\perp,\nabla P_{\leq \lambda_n}\psi]\Lambda z \rangle=\langle P_{\leq \lambda_n} :(\Lambda z\nabla^\bot z):
  , \nabla\psi\rangle$$ which holds true by approximation and in particular   no renormalization constant is needed.

Let now $n\rightarrow\infty$. Then $$\int_0^\cdot\langle
   P_{\leq \lambda_n}\Lambda^\alpha \psi,\dif B_s\rangle\to  \int_0^\cdot\langle
   \Lambda^\alpha \psi,\dif B_s\rangle\quad\mbox{in}\quad L^2(\Omega, C[0,\infty)).$$
 As $\theta_n\rightarrow \theta$  in $L^{2r}(\Omega,L^{p}_{\text{loc}}([0,\infty);B_{\infty,1}^{-1/2})\cap C([0,\infty);B^{-1/2-\delta}_{\infty,1}))$ and $q_n\to0$ in $L^r(\Omega,L^p_{\text{loc}}([0,\infty);L^\infty))$ for all $p\in[1,\infty)$ and $z_n\to z$ in $L^{2r}(\Omega, C([0,\infty);C^{\gamma_1/2-\alpha-\kappa})$, we apply Proposition~\ref{p:51} and \eqref{besov} we deduce that terms in \eqref{thetan} converge to their natural limits as in Definition~\ref{d:1} and hence $\theta$ satisfies \eqref{eq2} in the sense of Definition~\ref{d:1}. By \eqref{induction wf} and lower-semicontinuity we obtain
\begin{align}
\label{c1}
\$\theta-\Lambda z\$_{C^{1/2}_{t}B^{-1-\delta}_{\infty,1},2r}\leq M_L^{1/2}+3M_0M_L^{1/2}.
\end{align}
Also \eqref{induction w1} implies \eqref{eq:32}.

Next, we show that
\eqref{est:fn} follows from \eqref{induction fn}. By \eqref{induction fn} and \eqref{para} we know for some $\vartheta>1/2$
\begin{equation}
\label{eq:ee}
\sup_{s\geq4}\left(\mathbf{E}\sup_{t\in[s,s+1]}\|\theta-\Lambda z\|^{2r}_{B^{\vartheta-1}_{\infty,1}}\right)^{1/2r}\leq 3M_0M_L^{1/2}a^{\beta/4}+3M_L^{1/2}.
\end{equation}
Here, $\vartheta$ might be different from $\vartheta$ in \eqref{induction fn} and by \eqref{induction fn} we could always find such $\vartheta$.

Non-uniqueness of solutions as well as the existence of infinitely many solutions and even a  continuum of solutions follow from exactly the same argument as \cite[Sections 3.2.2-3.3.4]{HZZ22} except that we have to take expectation during the computation and apply \eqref{esta}.

If $\theta_0$ is non-Gaussian, we could choose $a$ large enough to guarantee that $P_{\leq\lambda_0}\theta_0$  is non-Gaussian. Accordingly, also the projection of the linear solution  $P_{\leq\lambda_0}z$ is non-Gaussian as $\theta_0$ is independent of the driving noise. This finally implies that also $\theta$ is non-Gaussian since $P_{\leq\lambda_0}f=0$.

If $\theta_0$ is Gaussian, then we can redefine our iteration starting from $  f_0  =\chi Z$, where $Z:\Omega\times \mT^2\to \mR$ is a smooth bounded non-Gaussian random variable with only finitely many non-zero Fourier modes, measurable with respect to $\cF_0$ and independent of the driving noise and from $\theta_0$, where we may need to augment $\cF_0$ and also the filtration $(\cF_t)$; and  $\chi:[0,\infty)\to \mR$ is a smooth cut-off function which equals to $0$ for $t\leq 1/2$ and equals to $1$ for $t>1$.  Then we could choose $\lambda_0$ large enough such that $P_{\leq \lambda_0}Z=Z.$
Thus \eqref{induction ps} on the level $n=0$ reads as
\[-\mathcal{R}Z\partial_t\chi +\chi^2\Lambda Z \nabla^\bot Z +\chi\Lambda z_0 \nabla^\bot Z+\chi\Lambda Z \nabla^\bot z_0+P_{\leq \lambda_0}(:\Lambda z \nabla^\bot z:+\Lambda^{\gamma_1}\mathcal{R}z-\Lambda^{\gamma}\mathcal{R}z)- \chi\Lambda^{\gamma}\mathcal{R} Z=: \nabla q_0  \]
Hence, for some $r\geq1$
\begin{align*}
	\$q_0\$_{X,r}\leq M_L.
\end{align*}
Then it follows from the proof of Proposition~\ref{p:iteration} that $\sum_{n\in\mathbb{N}}\ell_n\leq 1$ and consequently for  $t\geq2$ $\hat{f}_{n}(k)=\hat{Z}(k)$ for $|k|\leq \lambda_0$  which implies that $\hat{\theta}(k)=\widehat{(\Lambda z+\Lambda Z)}(k)$ for $|k|\leq \lambda_0$. Thus $\theta$ is non-Gaussian by independence and the result follows.
\end{proof}

\section{Stationary and ergodic stationary solutions}
\label{s:in}

In this section, we study the long time behavior of solutions to \eqref{eq1} constructed in the previous sections. In particular, we show that they generate stationary as well as ergodic stationary solutions. Moreover, non-uniqueness of both types of stationary solutions  holds.

As it is common in the context of equations without uniqueness, we understand stationarity in terms of shift operators on trajectories instead of the more classical framework of Markov semigroups.
We define the trajectory space for the solution and the noise as $\cT:=C(\mR;B^{-1/2}_{p,1})\times C(\mR;B^{-1-\kappa}_{p,p})$ for some $p\in[2,\infty)$ and $\kappa>0$ and the shifts $S_t$, $t\in\mR$ on this space are given by
$$
S_t(\theta,B)(\cdot)=(\theta(\cdot+t),B(\cdot+t)-B(t))\quad t\in \mR,\quad (\theta,B)\in\cT.
$$
Formally $\partial_{t } B=\xi$ and $B$ represents the noise part in \eqref{eq1}. The set of probability measures on $\mathcal{T}$ is denoted by $\mathcal{P}(\mathcal{T})$. We have all in hand to formulate the notion of stationary solution.

\begin{definition}
We say that  $((\Omega,\mathcal{F},(\mathcal{F}_t),\mathbf{P}),\theta,B)$ is a stationary solution to \eqref{eq2} provided $B$ is a $(\mathcal{F}_t)$-cylindrical Wiener process on $L^2$ with zero spatial mean and $
(\theta,B)$ satisfies  \eqref{eq2} in the analytically weak sense on $(-\infty,\infty)$, more precisely
\begin{align*}
&\langle \theta(t),\psi\rangle+\int_s^t\frac12\langle \Lambda^{-1/2}\theta,\Lambda^{1/2}[\mathcal{R}^\perp\cdot,\nabla \psi] \theta\rangle\dif r
		\\&\qquad= \langle \theta(s),\psi\rangle - \int_{s}^{t}\langle\ \nu\Lambda^{\gamma } \theta
		,\psi\rangle\dif r+\int_{s}^{t}\langle \Lambda^\alpha\psi,\dif B_r
		\rangle,
	\quad\forall \psi \in C^{\infty}(\mathbb{T}^{2}),\  t\geq s,
\end{align*}
and its law is shift invariant, that is,
$$\mathcal{L}[S_{t}(\theta,B)]=\mathcal{L}[\theta,B]\qquad\text{ for all }\quad t\in\mR.$$
\end{definition}

Note that stationary solutions satisfy the equation on $\mathbb{R}$. Thus, in the first step, we shall extend the convex integration solutions from Theorem~\ref{thm:6.1} from $\mathbb{R}^{+}$ to $\mathbb{R}$. We simply define $\theta(t)=\theta(4)$ for $t\leq 4$ and we choose two sided Brownian motion for $B$ which is already defined on $\mathbb{R}$. For the sequel it is irrelevant that the equation is not satisfied for time $t\leq 4$ or that we neglected the initial value.
We made this choice in order to be able to apply \eqref{est:fn} uniformly on $\mR$.

A Krylov--Bogoliubov argument then shows that the approach of Section~\ref{s:1.1} leads to  convex integration solutions  which generate   stationary solutions as limits of their ergodic averages. However, it is convenient to replace the initial value problem solution $z$ by the unique stationary solution to the linear equation \eqref{eqz}, namely, to let
\begin{equation}\label{eq:zstat}
z(t)=\int_{-\infty}^te^{-(t-s)\Lambda^{\gamma_1}}\Lambda^{\alpha-1}\dif B_s.
\end{equation}
It satisfies the estimates of Proposition~\ref{fe z} and can be used to construct the renormalized product $:\Lambda z\nabla^{\perp}z:$ as in Proposition~\ref{lem:re}. These are the stochastic elements that we now employ in the iterative construction of Section~\ref{s:a}.

\bt\label{th:s1}
Let $\theta$ be a convex integration solution  starting from $\theta_{0}=0$ with $z$ given by \eqref{eq:zstat}, obtained through the iteration in Theorem~\ref{thm:6.1}
and extended to $\mR$ by the above construction. Then there exists a sequence $T_{n}\to\infty$ and
 a stationary  solution $((\tilde\Omega,\tilde{\mathcal{F}},(\tilde{\mathcal{F}}_t),\tilde{\mathbf{P}}),\tilde\theta,\tilde B)$ to \eqref{eq1} such that for some $\tau\geq0$
 $$
 \frac{1}{T_{n}}\int_{0}^{T_{n}}\mathcal{L}[S_{t+\tau}(\theta,B)] \dif t\to \mathcal{L}[\tilde\theta,\tilde B] \quad\mbox{weakly in}\quad \mathcal{P}(\mathcal{T})\quad\mbox{as}\quad n\to\infty.
 $$
Moreover, it holds true
some $\delta\in (0,1),\kappa>0$ small enough
 \begin{equation}\label{eq:s}
 \$\tilde\theta\$_{B^{-1/2+\delta}_{\infty,1},2r}+\$\tilde\theta\$_{C^{{1/2-\kappa}}_tB^{-1-\delta}_{\infty,1},2r}\lesssim 1.
 \end{equation}
\et

\begin{proof}
It follows from  \eqref{est:fn} and the construction that there
	exists $\delta\in(0,1)$, $\kappa>0$  so that for every $N\in\mN$
\begin{align}\label{eq:B1}
&\sup_{s\in\mR}\mathbf{E} \left[\| \theta(\cdot+s) \|_{C([-N,N];B_{\infty, 1}^{- 1 / 2 +
				\delta})}+\| \theta(\cdot+s) \|_{C^{1/2-\kappa}([-N,N]; B^{- 1- \delta}_{\infty, 1})}\right]
\\\lesssim &\sup_{s\in\mR}\sum_{i=-N}^N\mathbf{E} \left[\| \theta(\cdot+s) \|_{C([i,i+1];B_{\infty, 1}^{- 1 / 2 +
				\delta})}+\| \theta(\cdot+s) \|_{C^{1/2-\kappa}([i,i+1]; B^{- 1- \delta}_{\infty, 1})}\right]\nonumber
\\\lesssim &N.\nonumber
				\end{align}
		Next, for $T\geq0$ and $\tau\geq0$ we define the following ergodic average as the probability measure on $\cT$
	\begin{align*}
		\nu_{\tau,T}:=\frac1T\int_0^T\cL[S_{t+\tau}(\theta,\zeta)]\dif t,
	\end{align*}
	and we show that the family $\nu_{\tau,T}$, $T\geq 0$, is tight. To this end,
we define for $R>0$ and $\kappa>0$
	$$B_R:=\cap_{N=M}^\infty\Big\{g;\,\|g\|_{C^{1/2-\kappa}([-N,N]; B^{- 1 - \delta}_{\infty, 1})}+\|g\|_{C([-N,N];
		B_{\infty, 1}^{- 1 / 2 + \delta})}\leq R_N\Big\},$$
	which is relatively compact in $C(\mR;B_{p,1}^{-1/2})$.
	Since $S_tB$ is a Wiener process for every $t\in \mR$, the law of $S_tB$ does not change with $t\in \mR$ and is tight. Then we use \eqref{eq:B1} to deduce that  there exists a compact set $K_\eps$ in $\mathcal{T}$ such that
	$$\sup_{t\in\mR}\mathbf{P}(S_t(\theta,B)\in K_\eps^c)<\eps.$$
	This implies
	\begin{align*}
		\nu_{\tau,T} (K_\eps^c) &=  \frac{1}{T} \int_0^T \mathbf{P}(S_{t+\tau} (\theta,B)\in K_\eps^c)  \dif t <\eps.
	\end{align*}
	Now we could use excatly the same argument as \cite[Theorem 4.2]{HZZ22} to obtain the result.
\end{proof}

\begin{corollary}\label{cor:4.3}
There are infinitely many non-Gaussian stationary solutions.
\end{corollary}

\begin{proof}
For the non-uniqueness of stationary solutions, we start  the iteration in Proposition \ref{p:iteration} with
$   f_0   = Z.$ Here,  $Z$ is a smooth deterministic function with only finitely many non-zero Fourier modes such that we could choose $\lambda_0$ large enough to have $P_{\leq \lambda_0}Z=Z.$
Thus \eqref{induction ps} on the level $n=0$ reads as
\[ \Lambda Z \nabla^\bot Z +\Lambda z_0 \nabla^\bot Z+\Lambda Z \nabla^\bot z_0+P_{\leq \lambda_0}(:\Lambda z \nabla^\bot z:+\Lambda^{\gamma_1}\mathcal{R}z-\Lambda^{\gamma}\mathcal{R}z)- \Lambda^{\gamma}\mathcal{R} Z=: \nabla q_0  \]
Hence, for some $r\geq1$
\begin{align*}
	\$q_0\$_{X,r}\leq M_L.
\end{align*}
Now, we run the convex integration based on Proposition \ref{p:iteration} and we get a limit $f = \lim_{n\to\infty}   f_n  $.
Thus, 	as in Proposition \ref{p:iteration} we could choose $a$ large enough so that  \eqref{induction w1} holds which implies that the solution $\theta$ obtained in Theorem \ref{thm:6.1} extended to $\mR$ as explained above satisfies
$$\sup_{s\geq4}\mathbf{E}\sup_{t\in[s,s+1]}\|\theta-\Lambda z-\Lambda Z\|_{B^{-1/2-\delta}_{\infty,1}}\leq \eps.$$
Then  for any $t>0$ and $\tau\geq4$,  $\mathbf{E}^{\nu_{\tau,T-\tau}}\|\theta(t)-\Lambda z(t)-\Lambda Z\|_{B^{-1/2-\delta}_{\infty,1}}\leq \eps$.
As in the proof of Theorem~\ref{th:s1} we have tightness of $(\theta,B)$ and we modify the stochastic basis to construct a new stochastic process $(\tilde \theta,\tilde B)$ as a stationary solution to \eqref{eq2}. Taking the limit we obtain by lower-semicontinuity $$\tilde{\mathbf{E}}\|\tilde\theta(t)-\Lambda\tilde z(t)-\Lambda Z\|_{B^{-1/2-\delta}_{\infty,1}}\leq \eps.$$
Here $\tilde z(t)=\int_{-\infty}^t e^{-(t-s)\Lambda^{\gamma_1}}\Lambda^{\alpha-1} \dif \tilde B_s$. The convergence of the $z$ part follows from
$$\tilde z(t)=\Lambda^{\alpha-1} \tilde B_t-\int_{-\infty}^t e^{-(t-s)\Lambda^{\gamma_1}}\Lambda^{\gamma_1+\alpha-1}  \tilde B_s\dif s$$ and convergence of $\tilde B$ part.

In order to prove non-Gaussianity, we choose $Z$ to be  a bounded random variable, measurable with respect to $\cF_0$, independent of the driven noise and such that  $P_{\leq\lambda_0}Z$ is non-Gaussian, where we may need to augment $\cF_0$ and also the filtration $(\cF_t)$. The same argument as Theorem~\ref{thm:6.1} implies that for $|k|\leq\lambda_0$ $\hat{\theta}(k)=\widehat{(\Lambda{z}+\Lambda{Z})}(k)$  is non-Gaussian. As this holds for any $t>0$ and $z$ is stationary and $Z$ independent of $t$ it follows that the law of $\hat{\theta}_t(k)$ does not change over time which implies that $\tilde\theta$ is non-Gaussian. 
Non-uniqueness is achieved by choosing different $Z$.
\end{proof}

We proceed with the definition of ergodic stationary solution which coincides with the notion of ergodicity of the dynamical system
$ \big(\mathcal{T}, {\rm Borel} (\mathcal{T}), (S_t, t \in \mR), \mathcal{L} [{\theta},B]\big).$

\begin{definition}
A stationary solution $((\Omega,\mathcal{F},(\mathcal{F}_t),\mathbf{P}),\theta,B)$ is ergodic provided
\[ \mathcal{L}[\theta,B] (A) = 1 \quad \text{or} \quad \mathcal{L}[\theta,B] (A) = 0 \quad \text{for all } A
\subset\mathcal{T} \text{ Borel and shift invariant}. \]
\end{definition}

As a consequence of Theorem \ref{th:s1} we obtain their existence, non-uniqueness and non-Gaussianity.

\bt\label{thm:4.5}
There exist $K>0$ and an ergodic stationary solution $((\Omega,\mathcal{F},(\mathcal{F}_t),\mathbf{P}),\theta,B)$ satisfying for some $\delta\in(0,1),\kappa>0$
\begin{equation}\label{eq:s55}
 \$\theta\$_{B^{-1/2+\delta}_{\infty,1},2r}+\$\theta\$_{C_t^{1/2-\kappa}B^{-1-\delta}_{\infty,1},2r}\leq K,
 \end{equation}
 and it holds
\begin{equation}\label{eq:s56}\$\theta-\Lambda z-\Lambda Z\$_{B^{-1/2-\delta}_{\infty,1},2r}\leq 1.\end{equation}
Here  $z=\int_{-\infty}^t e^{-(t-s)\Lambda^{\gamma_1}} \Lambda^{\alpha-1}\dif B_s$ and $Z$ is a smooth deterministic function with only finitely many non-zero Fourier modes.
In particular, there are infinitely many non-Gaussian ergodic solutions.
\et

\begin{proof}
The result is obtained by a classical Krein--Milman argument and
 Theorem~\ref{th:s1}. Indeed, the set of all laws of  stationary solutions satisfying \eqref{eq:s55} and \eqref{eq:s56}
is non-empty, convex, tight and closed which follows from the same argument as the proof of Theorem \ref{th:s1} and Corollary \ref{cor:4.3}. Hence there exists an extremal point, which is  an ergodic stationary solution. Non-uniqueness is achieved by choosing different $Z$.
Non-Gaussianity follows from the fact that the stationary solution satisfies for $|k|\leq\lambda_0$   $\hat\theta(k)=\widehat{(\Lambda z+\Lambda Z)}(k)$ with $Z$  as in Corollary \ref{cor:4.3} also form a  non-empty, convex, tight and closed set.
\end{proof}

\section{Extension to fractional Navier--Stokes equations}\label{other}

Let us consider a singular fractional Navier--Stokes equations on $\mathbb{T}^2$. The equations govern the time evolution of the fluid velocity $u$ and  read as
\begin{equation}
\label{1}
\aligned
 \partial_{t} u+\div(u\otimes u)+\nabla P&=-\nu\Lambda^\gamma u +\Lambda^{-1+\alpha} \xi,
\\
\div u&=0,\\
u(0)&=u_{0}.
\endaligned
\end{equation}
Here $P$ is the associated pressure, $\xi$ is the space-time white noise with mean zero and $\nu\geq0$, $\alpha\in[0,1/2)$, $\gamma\in[0,1)$. We choose the exponent in the derivative of the noise this way in order to compare our result to \cite[Example 5.1 iii.]{GP23}. In the latter work, the authors considered the 2D Navier--Stokes equations in the vorticity form with noise $\Lambda^{\alpha }\xi$ and $\alpha=\gamma/2$. Going back to the velocity level therefore gains one derivative and leads to the formulation \eqref{1}. The equation in \cite{GP23} formally possesses a Gaussian invariant measure, which can be used to obtain existence and uniqueness for $\gamma\geq 1$ and existence for $\gamma\in (0,1)$. Similarly to the SQG equation \eqref{eq2} we obtain non-uniqueness but here we can cover the full supercritical regime. As a matter of fact, our approach  also applies to the Euler setting $\nu=0$.

We decompose the velocity as $u=z+v$
 where
 \begin{align}\label{eqz1}\partial_t z +\nabla P_z= -\Lambda^{\gamma_1}  z+ \Lambda^{- 1  {+\alpha}} \xi,\quad \div z =0,\quad
    z(0)=u_0,\end{align}
 \begin{align}\label{eqv1}  \partial_t v -\Lambda^{\gamma_1}  z+\nu\Lambda^{\gamma} z+ \div(v+z)\otimes(v+z) +\nabla P_v=
  -  \nu\Lambda^{\gamma} v, \quad \div v =0,\quad
    v(0)=0,\end{align}
    with  an  additional parameter $\gamma_{1}\in (2\alpha, 2-2\alpha)$.
We can apply the convex integration from \cite{CL20, CL22} to obtain existence and non-uniqueness of global solutions for every $u_0\in C^{1-2\alpha}$ satisfying the  divergence free condition. Although the construction  \cite{CL20, CL22, HZZ19, HZZ22a, HZZ21a} is formulated for $\gamma=1$, it also holds for $\gamma<1$ since the fractional Laplacian is easier to  control and the rest of the proof does not change. In fact, by Proposition \ref{fe z} it follows that  $z\in C_TC^{\gamma_1/2-\alpha-\kappa}$. To use convex integration method we require that $\Lambda^{\gamma_1}z$, $\Lambda^\gamma z\in C^{-1+}$, i.e. $\gamma_1<2-2\alpha$, $\gamma<2-2\alpha$. Then for these terms $\Lambda^{\gamma_1}z$, $\Lambda^\gamma z$ we can use a similar argument as for the SQG equation in this paper. Moreover, we need that $z$ is a function in $L^2$ which requires $\gamma_1>2\alpha$. This leads to our condition $2\alpha<\gamma_1<2-2\alpha$ and $\alpha<1/2$. For the error part containing the $z$ term we can use a similar argument as in \cite{LZ23, CDZ22} since $z$ is a function and the treatment of the  other terms is similar to \cite{CL20, CL22}. To summarize, we cover the supercritical regime in the velocity form of \cite[Example 5.1 iii.]{GP23} and obtain the following result.

 \bt\label{thm:5.1}
 Fix $\gamma\in [0,2-2\alpha)$, $\alpha\in [0,1/2)$.
For any divergence free  initial condition $u_0\in C^{1-2\alpha}$ $\mathbf{P}$-a.s.
there exists an $(\mathcal{F}_{t})$-adapted analytically weak solution $u$ to \eqref{1} with $u(0)=u_0$ which belongs to $L^p_{\rm{loc}}([0,\infty);L^2)\cap C([0,\infty);W^{-1,1})$ $\mathbf{P}$-a.s. for all $p\in[1,\infty)$. There are infinitely many such solutions $u$.

\et

\begin{remark}
In comparison to the SQG situation,  the Navier--Stokes setting  is actually easier. The reason is the different form of the nonlinearity in the equation where convex integration is applied: $\Lambda z\nabla^{\perp}z$ compared to $\div (z\otimes z)$. Particularly, the Navier--Stokes nonlinearity is well-defined for $z$ merely function valued whereas the SQG setting requires more regularity of $z$. If $\gamma$ is small then $z$ defined via the conventional Da Prato--Debussche trick with $\gamma_{1}=\gamma$ is not function valued and the method does not work. Making $\gamma_{1}>\gamma$ makes $z$ more regular which can be nicely combined with convex integration to treat even supercritical/critical situations.
\end{remark}

\appendix
\renewcommand{\appendixname}{Appendix~\Alph{section}}
\renewcommand{\theequation}{A.\arabic{equation}}

\section{Some auxiliary results}
\label{s:B}

In this part, we recall some auxiliary lemmas from \cite{CKL21}. The first result is rather classical and permits to rewrite the nonlinearity and  formulate the notion of weak solution as in Definition~\ref{d:1}.

\begin{proposition}{\cite[Proposition 5.1]{CKL21}}\label{p:51}
Let $\mathcal{R}=\mathcal{R}_{j}$, $j=1,2$. Assume that $\phi\in H^{3}$ and $\theta\in \dot{H}^{-1/2}$. Then we have
$$
\|[\mathcal{R},\phi]\theta\|_{\dot{H}^{1/2}}\lesssim \|\phi\|_{\dot{H}^{3}}\|\theta\|_{\dot{H}^{-1/2}}.
$$
\end{proposition}

The second result is a kind of Leibniz rule useful in the estimate of the new stress in Section~\ref{s:estq}.

\begin{lemma}{\cite[Lemma 2.1]{CKL21}}\label{lem:21} Let $|l|=1,\lambda l\in\mathbb{Z}^2$, and $g(x)=a(x)\cos (\lambda l\cdot x).$ Then
$$\Lambda g=\lambda g+l\cdot \nabla a \sin(\lambda l\cdot x)+ (T^{(1)}_{\lambda l}a) \cos(\lambda l\cdot x)+(T^{(2)}_{\lambda l}a) \sin(\lambda l\cdot x),$$
where $$\widehat{T_{\lambda l}^{(1)}a}(k)=\bigg(\frac{|\lambda l+k|+|\lambda l-k|}2-\lambda\bigg)\hat{a}(k),$$
$$\widehat{T_{\lambda l}^{(2)}a}(k)=i\bigg(\frac{|\lambda l+k|-|\lambda l-k|}2-l\cdot k\bigg)\hat{a}(k).$$
\end{lemma}

Finally, in order to control the operators in Lemma~\ref{lem:21}, we rely on the following.
	
\bl{\cite[Lemma 3.3]{CKL21}}\label{lem:33}   Assume $b_0:\mT^2\to \mR$ with $\supp(\hat{b}_0)\subset\sup\{|k|\leq \mu\}$ and $10\leq \mu\leq \frac12\lambda$.
Then  $$\|T_{\lambda l}^{(1)}b_0\|_{L^\infty}\lesssim \lambda^{-1}\mu^2\|b_0\|_{L^\infty},$$
$$\|T_{\lambda l}^{(2)}b_0\|_{L^\infty}\lesssim \lambda^{-2}\mu^3\|b_0\|_{L^\infty}.$$
\el
%

%

\def\cprime{$'$} \def\ocirc#1{\ifmmode\setbox0=\hbox{$#1$}\dimen0=\ht0
  \advance\dimen0 by1pt\rlap{\hbox to\wd0{\hss\raise\dimen0
  \hbox{\hskip.2em$\scriptscriptstyle\circ$}\hss}}#1\else {\accent"17 #1}\fi}

\end{document}